\documentclass[letterpaper, 11pt,  reqno]{amsart}

\usepackage{amsmath,amssymb,amscd,amsthm,amsxtra, esint}
\usepackage{mathabx}

\usepackage[margin=1.2in,marginparwidth=1.5cm, marginparsep=0.5cm]{geometry}

\usepackage[mathscr]{euscript}

\usepackage[implicit=true]{hyperref}

\usepackage{dsfont}%for indicator funct.

\allowdisplaybreaks[2]

\usepackage{cancel}

\usepackage{color}

\definecolor{gr}{rgb}   {0.,   0.69,   0.23 }
\definecolor{bl}{rgb}   {0.,   0.5,   1. }
\definecolor{mg}{rgb}   {0.85,  0.,    0.85}
\definecolor{yl}{rgb}   {0.8,  0.7,   0.}
\definecolor{or}{rgb}  {0.7,0.2,0.2}

\newtheorem{theorem}{Theorem} [section]

\newtheorem{lemma}[theorem]{Lemma}
\newtheorem{proposition}[theorem]{Proposition}
\newtheorem{remark}[theorem]{Remark}

\newtheorem{definition}[theorem]{Definition}
\newtheorem{corollary}[theorem]{Corollary}

\DeclareMathOperator*{\supp}{supp}

\def\11{{\rm 1~\hspace{-1.4ex}l} }
\def\R{\mathbb R}

\def\Z{\mathbb Z}
\def\N{\mathbb N}

\def\T{\mathbb T}

\newcommand{\Q}{\mathbf{Q}}

\renewcommand{\l}{\ell}

\newcommand{\pa}{\partial}

\usepackage{tikz}

\usetikzlibrary{shapes.misc}
\usetikzlibrary{shapes.symbols}
\usetikzlibrary{shapes.geometric}
\tikzset{
	dot/.style={circle,fill=black,draw=black,inner sep=0pt,minimum size=0.5mm},
	>=stealth,
	}
\tikzset{
	ddot/.style={circle,fill=white,draw=black,inner sep=0pt,minimum size=0.8mm},
	>=stealth,
	}

\tikzset{decision/.style={ % requires library shapes.geometric
        draw,
        diamond,
        aspect=1.5
    }}

\tikzset{dia2/.style
={diamond,fill=white,draw=black,inner sep=0pt,minimum size=1mm},
	>=stealth,
	}

\tikzset{dia/.style
={star,fill=black,draw=black,inner sep=0pt,minimum size=1mm},
	>=stealth,
	}

\makeatletter
\def\DeclareSymbol#1#2#3{\expandafter\gdef\csname MH@symb@#1\endcsname{\tikz[baseline=#2,scale=0.15]{#3}}}
\def\<#1>{\csname MH@symb@#1\endcsname}
\makeatother

\DeclareSymbol{X}{-2.4}{\node[dot] {};}
\DeclareSymbol{1}{0}{\draw[white] (-.4,0) -- (.4,0); \draw (0,0)  -- (0,1.2) node[dot] {};}
\DeclareSymbol{2}{0}{\draw (-0.5,1.2) node[dot] {} -- (0,0) -- (0.5,1.2) node[dot] {};}

\DeclareSymbol{1}{-2.7}
 {
  \draw (0,0) node[dot]{};}

\DeclareSymbol{1'}{-2.7}
 {\draw (0,0) node[ddot]{};}

\DeclareSymbol{1''}{-2.7}
 {\draw (0,0) node[dia]{};}

\DeclareSymbol{3}{0}
 {\draw (0,0) node[dot]{} -- (0,1.2) node[ddot] {}; 
 \draw (-1.2,0) node[dot] {} -- (0,1.2)node[ddot] {} -- (1.2,0) node[dot] {};}

\DeclareSymbol{3'}{0}
 {\draw (0,0) node[dia]{} -- (0,1.2) node[ddot] {}; 
 \draw (-1.2,0) node[dia] {} -- (0,1.2)node[ddot] {} -- (1.2,0) node[dia] {};}

\DeclareSymbol{31}{-3}{
\draw (0,-1)node[dot] {} -- (0,0) node[ddot] {}-- (0.9, -1) node[dot] {}; 
\draw (0,-1)node[dot] {} -- (0,0) node[ddot] {}-- (-0.9, -1) node[dot] {}; 
\draw (0,0)node[ddot] {} -- (1.3,1) node[ddot] {}-- (1.3, 0) node[dot] {}; 
\draw  (1.3,1) node[ddot] {}-- (2.6, 0) node[dot] {}; 
}

\DeclareSymbol{32}{-3}{
\draw (0,-1)node[dot] {} -- (0,0) node[ddot] {}-- (0.9, -1) node[dot] {}; 
\draw (0,-1)node[dot] {} -- (0,0) node[ddot] {}-- (-0.9, -1) node[dot] {}; 
\draw (0,0)node[ddot] {} -- (0,1) node[ddot] {}-- (0.9, 0) node[dot] {}; 
\draw (0,0)node[ddot] {} -- (0,1) node[ddot] {}-- (-0.9, 0) node[dot] {}; 
}

\DeclareSymbol{33}{-3}{
\draw (2.6,-1)node[dot] {} -- (2.6,0) node[ddot] {}-- (3.5, -1) node[dot] {}; 
\draw (2.6,-1)node[dot] {} -- (2.6,0) node[ddot] {}-- (1.7, -1) node[dot] {}; 
\draw (0,0)node[dot] {} -- (1.3,1) node[ddot] {}-- (1.3, 0) node[dot] {}; 
\draw  (1.3,1) node[ddot] {}-- (2.6, 0) node[ddot] {}; 
}

\DeclareSymbol{3''}{0}
 {\draw (0,0) node[dia]{} -- (0,1.2) node[ddot] {}; 
 \draw (-1.2,0) node[dot] {} -- (0,1.2)node[ddot] {} -- (1.2,0) node[dia] {};}

\DeclareSymbol{3'''}{0}
 {\draw (0,0) node[dot]{} -- (0,1.2) node[ddot] {}; 
 \draw (-1.2,0) node[dot] {} -- (0,1.2)node[ddot] {} -- (1.2,0) node[dia] {};}

\DeclareSymbol{31'}{-3}{
\draw (0,-1.2)node[dia] {} -- (0,0) node[ddot] {}-- (0.9, -1.2) node[dia] {}; 
\draw (0,-1.2)node[dia] {} -- (0,0) node[ddot] {}-- (-0.9, -1.2) node[dia] {}; 
\draw (0,0)node[ddot] {} -- (1.3,1.2) node[ddot] {}-- (1.3, 0) node[dia] {}; 
\draw  (1.3,1.2) node[ddot] {}-- (2.6, 0) node[dia] {}; 
}

\DeclareSymbol{32'}{-3}{
\draw (0,-1.2)node[dia] {} -- (0,0) node[ddot] {}-- (0.9, -1.2) node[dia] {}; 
\draw (0,-1.2)node[dia] {} -- (0,0) node[ddot] {}-- (-0.9, -1.2) node[dia] {}; 
\draw (0,0)node[ddot] {} -- (0,1.2) node[ddot] {}-- (0.9, 0) node[dia] {}; 
\draw (0,0)node[ddot] {} -- (0,1.2) node[ddot] {}-- (-0.9, 0) node[dot] {}; 
}

\DeclareSymbol{33'}{-3}{
\draw (2.6,-1.2)node[dia] {} -- (2.6,0) node[ddot] {}-- (3.5, -1.2) node[dia] {}; 
\draw (2.6,-1.2)node[dia] {} -- (2.6,0) node[ddot] {}-- (1.7, -1.2) node[dia] {}; 
\draw (0,0)node[dot] {} -- (1.3,1.2) node[ddot] {}-- (1.3, 0) node[dot] {}; 
\draw  (1.3,1.2) node[ddot] {}-- (2.6, 0) node[ddot] {}; 
}

\newtheorem*{ackno}{Acknowledgments}

\numberwithin{equation}{section}
\numberwithin{theorem}{section}

\makeatletter
\@namedef{subjclassname@2020}{%
  \textup{2020} Mathematics Subject Classification}
\makeatother

\begin{document}
\baselineskip = 14pt

\title[Critical HNLS on $\mathbb{R}\times \mathbb{T}$]{Hyperbolic nonlinear Schr\"odinger equations on $\mathbb{R}\times \mathbb{T}$}

\author[E.~Ba\c{s}ako\u{g}lu, C. Sun, N. Tzvetkov and Y.~Wang]
{Engin Ba\c{s}ako\u{g}lu, Chenmin Sun, Nikolay Tzvetkov and Yuzhao Wang}

\address{Engin Ba\c{s}ako\u{g}lu\\Institute of Mathematical Sciences\\ ShanghaiTech University\\ Shanghai\\ 201210\\ China}
\email{ebasakoglu@shanghaitech.edu.cn}

\address{Chenmin Sun\\ CNRS\\ Université Paris-Est Créteil\\ Laboratoire d’Analyse et de Mathématiques appliquées\\
UMR 8050 du CNRS\\ 94010 Créteil cedex\\ France}
\email{chenmin.sun@cnrs.fr}

\address{Nikolay Tzvetkov\\ Ecole Normale Supérieure de Lyon\\ UMPA\\ UMR CNRS-ENSL 5669\\ 46\\ allée d’Italie\\ 69364-Lyon Cedex 07\\ France}
\email{nikolay.tzvetkov@ens-lyon.fr}

\address{Yuzhao Wang\\School of Mathematics\\ University of Birmingham\\ Watson Building\\ Edgbaston\\ Birmingham \\B15 2TT\\ United Kingdom}
\email{y.wang.14@bham.ac.uk}



\dedicatory{Dedicated to Professor Yoshio Tsutsumi   on the
occasion of his seventieth birthday}

\begin{abstract}
In this paper, we consider the hyperbolic nonlinear Schr\"odinger equations (HNLS) on $\mathbb{R}\times\mathbb{T}$. We obtain the sharp local well-posedness up to the critical regularity for cubic nonlinearity and in critical spaces for higher odd nonlinearities. Moreover, when the initial data is small, we prove the global existence and scattering for the solutions to HNLS with higher nonlinearities (except the cubic one) in critical Sobolev spaces. The main ingredient of the proof is the sharp up to the endpoint local/global-in-time Strichartz estimates.
\end{abstract}

%\date{\today}

%\vspace*{-5mm}

%%
%
\maketitle

\vspace{-3mm}

\tableofcontents

\section{Introduction}
We consider the hyperbolic nonlinear Schr\"odinger equations (HNLS) on $\mathcal{M}:=\mathbb{R}\times\mathbb{T}$, $k\in\mathbb{Z}_+$:
\begin{equation}
\label{eq:HNLS}
\begin{cases}
i\pa_t u + \Box u =\pm|u|^{2k}u,\\
u(t,\mathbf{x})|_{t=0}=u_0(\mathbf{x}),
\end{cases}
\quad (t,\mathbf{x})\in\mathbb{R}\times\mathcal{M},
\end{equation}
where $\mathbf{x}=(x,y)\in \mathbb{R}\times\mathbb{T}$, and $\Box = \Box_{\mathbf{x}}= \partial_x^2-\partial_y^2$. HNLS in \eqref{eq:HNLS} arises in the study of gravity water waves, \cite{Totz1, Totz2}; it also appears naturally in the study of the hyperbolic-elliptic Davey-Stewartson system; see, for example, \cite{GS90, LP93, G13}. The mathematical theory for HNLS was initiated by studies \cite{GS1, GS2, GS90, LP93}, also over the last decade, the Cauchy problem \eqref{eq:HNLS} has been studied extensively in different settings; we refer to \cite{GT12,G13,YW13_2,MT15, Totz3, BD17, DMPS18,T20,BOW25,BW25}. We also mention the recent review \cite{S24}, where Saut and Wang proposed some conjectures and open questions for the HNLS equation.

When the problem \eqref{eq:HNLS} is posed on $\mathbb{R}^2$, the small data global well-posedness is known to hold in the critical spaces $H^{s_{2,k}}(\mathbb{R}^2)$, where
\begin{align}\label{criticalexp}
    s_{2,k}:=1-\frac{1}{k}.
\end{align}
The key ingredient of the proof of this fact is the Strichartz estimate
 \begin{align}\label{strichartzeuclid}
   \Vert e^{it\partial_x\partial_y}f\Vert_{L^q_tL^p_{x,y}(\mathbb{R}\times \mathbb{R}^2)}\lesssim \Vert f\Vert_{L^2(\mathbb{R}^2)} 
 \end{align}
for $(p,q)$ satisfying the conditions
\begin{align}\label{strichartzeuclidcond}
    2\leq p,q\leq \infty,\quad \frac{2}{q}+\frac{2}{p}=1,\quad q\neq 2.
\end{align}
The inequality \eqref{strichartzeuclid}-\eqref{strichartzeuclidcond} comes from the $TT^*$ method, Hardy-Littlewood-Sobolev inequality, and the dispersive inequality
\begin{align*}
  \Vert e^{it\partial_x\partial_y}f\Vert_{L^{\infty}_{x,y}}\lesssim \frac{1}{|t|}\Vert f\Vert_{L^1(\mathbb{R}^2)}.    
\end{align*}
See \cite{KT98} for a proof. Regarding the long-term dynamics on $\mathbb{R}^2$, Dodson-Marzuola-Pausader-Spirn \cite{DMPS18} proved the profile decomposition for the hyperbolic Schr\"odinger equation in both the mass-critical and the mass-supercritical cases, where the key element in their argument is the improved Strichartz estimate for the linear group $e^{it\partial_x\partial_y}$. As an application of the profile decomposition, they derived the existence of a minimal blow-up solution for the mass-critical problem. Furthermore, in \cite{Totz3}, Totz discussed an approach to global existence for sufficiently regular solutions to the Cauchy problem \eqref{eq:HNLS} posed on $\mathbb{R}^2$. 

Note that the estimate \eqref{strichartzeuclid}-\eqref{strichartzeuclidcond} is certainly the same for the semi-group $e^{it\Box}$, yet, when $(x,y)\in \R\times\T\,\,\text{or}\,\,\T\times\T$, the behavior of two such hyperbolic flows is quite different in terms of Strichartz estimates, which we will discuss later in this section. As mentioned above, by standard arguments, using \eqref{strichartzeuclid} we can show that the HNLS problem \eqref{eq:HNLS} is globally well-posed for small initial data in $H^{s_{2,k}}(\mathbb{R}^2)$, where $s_{2,k}$ is as in \eqref{criticalexp}. In this paper, we extend this result to small data global well-posedness in critical spaces $H^{s_{2,k}}(\mathbb{R}\times \mathbb{T})$, even if we do not have \eqref{strichartzeuclid} at our disposal in the $\mathbb{R}\times \mathbb{T}$ setting. Our main result in this paper reads as follows:
\begin{theorem}\label{maintm}
\begin{enumerate}
\mbox{}
        \item[$(i)$]$($Local well-posedness$)$ For $k\geq2$ fixed, the Cauchy problem \eqref{eq:HNLS} with data $u_0\in H^{1-\frac{1}{k}}(\mathbb{R}\times\mathbb{T})$ is locally well-posed. Moreover, when $k=1$, given $s>0$, the Cauchy problem \eqref{eq:HNLS} is locally well-posed for data $u_0\in H^{s}(\mathbb{R}\times\mathbb{T})$. 
\item[$(ii)$] $($Small data global well-posedness$)$ Let $k\geq2$ be fixed. Then, there exists $\epsilon>0$ such that for all initial data $u_0$ with $\Vert u_0\Vert_{H^{1-\frac{1}{k}}(\mathbb{R}\times \mathbb{T})}\leq \epsilon$, the solution $u$ to the Cauchy problem \eqref{eq:HNLS} is unique and exists globally-in-time. Moreover, for every such small initial data $u_0$, there exists $u_{\pm}$ such that
\begin{align*}
    \lim_{t\rightarrow\pm\infty}\Vert u-e^{it\Box}u_{\pm}\Vert_{H^{1-\frac{1}{k}}}=0.
\end{align*}
\end{enumerate}
\end{theorem}
\begin{remark}
Here we compare Theorem \ref{maintm} with the result \cite{TV12}. In \cite{TV12}, the third author and Visciglia studied global well-posedness and scattering for the cubic nonlinear Schr\"odinger equation on $\mathbb{R}^n\times M$, where $n\geq 2$ and $M$ is a compact Riemannian manifold of dimension $d\geq 1$. In particular, concerning the even dimensional case $($for $n\geq 2$ even$)$, given $\epsilon>0$ and small initial data, they proved the existence of global solutions in spaces $L^{\infty}_t\mathscr{H}_{x,y}^{\frac{n-2}{2},\frac{d}{2}+\epsilon}\cap X_{\epsilon}$ where
\begin{equation*}
\begin{aligned}
\Vert f\Vert_{\mathscr{H}_{x,y}^{s,r}}&=\sum_{|\alpha|\leq s}\Vert \partial_x^{\alpha}(1-\Delta)^{\frac{r}{2}}f\Vert_{L^2(\mathbb{R}^n_x\times M_y)}\\ \Vert f\Vert_{X_{\epsilon}}&=\sum_{s=0}^{\frac{n-2}{2}}\sum_{|\alpha|=s}\Vert \partial_x^{\alpha}(1-\Delta)^{\frac{1}{2}(\frac{d}{2}+\epsilon)}f\Vert_{L^4_tL_x^{\frac{2n}{1+2s}}L_y^2}.
\end{aligned}
\end{equation*}
The analysis performed in \cite{TV12} is based on the dimension $n\geq 2$ being odd or even, nevertheless using the strategy described in \cite{TV12}, an analogous result can also be obtained when $n=1$. In the following, we discuss this briefly for the hyperbolic NLS \eqref{eq:HNLS} with $k\geq 2$. First note that arguing as in \cite[Proposition 2.1]{TV12} one can obtain:
\begin{align*}
    \Vert e^{it\Box}f\Vert_{L^p_tL^q_xL^2_y}+ \Big\Vert \int_0^te^{i(t-\tau)\Box}F(\tau,x,y)d\tau\Big\Vert_{L^p_tL^q_xL^2_y}\lesssim \Vert f \Vert_{L^2_{x,y}}+\Vert F \Vert_{L^{p^{\prime}}_t L^{q^{\prime}}_x L^2_y}
\end{align*}
where $(t,x,y)\in\mathbb{R}\times\mathbb{R}\times \mathbb{T}$,\begin{align*}
  \frac{2}{p}+\frac{1}{q}=\frac{1}{2},\quad 4\leq p\leq \infty.  \end{align*}
Let $\widetilde{u}_j\in \{u_j, \overline{u}_j\}$. Then, given $\varepsilon>0$, implementing the Strichartz estimate above with $(p,q)=(2k+2,\frac{2k+2}{k-1})$, $k\geq 2$, leads to
\begin{equation}\label{estliketv12}
\begin{aligned}
    &\Vert e^{it\Box}f\Vert_{L^{2k+2}_tW^{\frac{1}{2}-\frac{1}{k},\frac{2k+2}{k-1}}_xH^{\frac{1}{2}+\varepsilon}_y}+ \Big\Vert \int_0^te^{i(t-\tau)\Box}\prod^{2k+1}_{j=1}\widetilde{u}_jd\tau\Big\Vert_{L^{2k+2}_tW^{\frac{1}{2}-\frac{1}{k},\frac{2k+2}{k-1}}_xH^{\frac{1}{2}+\varepsilon}_y}\\&\lesssim \Vert f \Vert_{H^{\frac{1}{2}-\frac{1}{k},\frac{1}{2}+\varepsilon}_{x,y}}+\Big\Vert \prod^{2k+1}_{j=1}\widetilde{u}_j \Big\Vert_{L^{\frac{2k+2}{2k+1}}_tW^{\frac{1}{2}-\frac{1}{k},\frac{2k+2}{k+3}}_x H^{\frac{1}{2}+\varepsilon}_y}\\&\lesssim \Vert f \Vert_{H^{\frac{1}{2}-\frac{1}{k},\frac{1}{2}+\varepsilon}_{x,y}}+\Big\Vert\langle \nabla_x\rangle^{\frac{1}{2}-\frac{1}{k}}\prod^{2k+1}_{j=1} \Vert\langle\nabla_y\rangle^{\frac{1}{2}+\varepsilon}u_j\Vert_{L^2_y}\Big\Vert_{L^{\frac{2k+2}{2k+1}}_{t}L^{\frac{2k+2}{k+3}}_x}\\&\lesssim\Vert f \Vert_{H^{\frac{1}{2}-\frac{1}{k},\frac{1}{2}+\varepsilon}_{x,y}}+\Bigg\Vert\sum_{j=1}^{2k+1}\Big\Vert\langle \nabla_x\rangle^{\frac{1}{2}-\frac{1}{k}}\Vert\langle\nabla_y\rangle^{\frac{1}{2}+\varepsilon}u_j\Vert_{L^2_y}\Big\Vert_{L^{\frac{2k+2}{k-1}}_x}\prod_{\substack{l=1\\l\neq j}}^{2k+1} \Big\Vert\Vert\langle\nabla_y\rangle^{\frac{1}{2}+\varepsilon}u_l\Vert_{L^2_y}\Big\Vert_{L^{k(k+1)}_x}\Bigg\Vert_{L^{\frac{2k+2}{2k+1}}_{t}}\\&\lesssim\Vert f \Vert_{H^{\frac{1}{2}-\frac{1}{k},\frac{1}{2}+\varepsilon}_{x,y}}+\sum_{j=1}^{2k+1}\Big\Vert \Vert u_j\Vert_{W^{\frac{1}{2}-\frac{1}{k},\frac{2k+2}{k-1}}_xH^{\frac{1}{2}+\varepsilon}_y}\prod_{\substack{l=1\\l\neq j}}^{2k+1}\Vert u_l\Vert_{W^{\frac{1}{2}-\frac{1}{k},\frac{2k+2}{k-1}}_xH^{\frac{1}{2}+\varepsilon}_y}\Big\Vert_{L^{\frac{2k+2}{2k+1}}_{t}}\\&\lesssim\Vert f \Vert_{H^{\frac{1}{2}-\frac{1}{k},\frac{1}{2}+\varepsilon}_{x,y}}+\sum_{j=1}^{2k+1}\Big[\Vert u_j\Vert_{L^{2k+2}_tW^{\frac{1}{2}-\frac{1}{k},\frac{2k+2}{k-1}}_xH^{\frac{1}{2}+\varepsilon}_y}\prod_{\substack{l=1\\l\neq j}}^{2k+1}\Vert u_l\Vert_{L^{2k+2}_tW^{\frac{1}{2}-\frac{1}{k},\frac{2k+2}{k-1}}_xH^{\frac{1}{2}+\varepsilon}_y}\Big]
\end{aligned}
\end{equation}
where we use the algebra property of $H^{s}(\mathbb{T})$ for $s>\frac{1}{2}$, fractional Leibniz rule for the $L^{\frac{2k+2}{k+3}}_x$-norm, and the embedding $W_x^{\frac{1}{2}-\frac{1}{k},\frac{2k+2}{k-1}}(\mathbb{R})\hookrightarrow L_x^{k(k+1)}(\mathbb{R})$. See Proposition \ref{keyproposition} for comparison with \eqref{estliketv12}. Utilizing the estimate \eqref{estliketv12} in a fixed point argument, one could prove the existence of a global solution $u(t,x,y)$ $($with small initial data$)$ in the class $$L^{\infty}_tH^{\frac{1}{2}-\frac{1}{k},\frac{1}{2}+\varepsilon}_{x,y}\cap L^{2k+2}_tW^{\frac{1}{2}-\frac{1}{k},\frac{2k+2}{k-1}}_xH^{\frac{1}{2}+\varepsilon}_y.$$
Therefore, we infer that the total number of derivatives in Theorem \ref{maintm} is $\varepsilon$ better than the result due to \cite{TV12}; however the discussion above demonstrates that the method in \cite{TV12} yields a better result concerning the $x$ derivatives.
\end{remark}
\begin{remark}
With the method of the current paper we cannot establish a global existence of the solutions for cubic HNLS \eqref{eq:HNLS}. However, by performing the modified scattering argument in \cite{HPTV15}, it is possible to obtain small data global well-posedness for cubic HNLS \eqref{eq:HNLS} on $\mathbb{R}\times\mathbb{T}$, which will be addressed in a further study. Moreover, we mention the recent work by Corcho and Mallqui \cite{CM25}, where they showed in particular that the cubic hyperbolic NLS given by the higher order operator $\partial_x^2+(-\partial_y^2)^{\alpha}$, $(x,y)\in\mathbb{R}\times\mathbb{T}$, with $\alpha>1$ is globally well-posed for small data in $L^2(\mathbb{R}\times\mathbb{T})$. The argument in \cite{CM25} depends essentially on the conservation of mass and crucial bilinear estimates that are proven in the sense of \cite{TT}.
\end{remark}
The nonlinear Schr\"odinger equations (NLS) have an extensive literature. For instance, concerning the $\mathbb{R}^d$ case, the construction of global $L^2$ solutions of NLS dates back to 1980's, see the classical result due to Tsutsumi \cite{Ts87} and the references therein. Moreover, the study of NLS on compact or partially compact manifolds has been ongoing for decades; see, for example, \cite{B93, BGT04, BGT051, BGT052, IP120, IP12,Herr13, HTT14}. In this direction, the first critical well-posedness result for NLS posed on a compact manifold was established by Herr, Tataru and the third author \cite{HTT11}, where they consider the energy critical problem on $\mathbb{T}^3$, and extend Bourgain's three dimensional subcritical result \cite{B93}. Later, the fourth author \cite{YW13} extended Bourgain's subcritical results in \cite{B93} to critical spaces and thus generalized the result in \cite{HTT11}. As for the semi-periodic case, Herr, Tataru and the third author \cite{HTT14} studied cubic NLS and obtained global well-posedness for small data in $H^s(M)$, $s\geq 1$, where $M=\mathbb{R}^m\times\mathbb{T}^n$, $0\leq m \leq 4$, $m+n=4$. We also refer to the works of Ionescu-Pausader \cite{IP120} and Hani-Pausader \cite{HP14}, in which they studied the long-time dynamics of the cubic and quintic NLS for data of arbitrary size in $H^1(M)$, $M=\mathbb{R}\times\mathbb{T}^3$, $\mathbb{R}\times\mathbb{T}^2$ respectively. 

We now state the local/global Strichartz estimates, which will be the main tool to prove Theorem \ref{maintm}. \begin{proposition}\label{epsilonremovalprop}
 For any $N\geq 1$ and $p> 4$, we have 
\begin{align}\label{BDstrcest2}
\Vert e^{it\Box}P_{\leq N}\phi\Vert_{L^p_{t,\mathbf{x}}([0,1]\times \mathcal{M})}\lesssim N^{1-\frac{4}{p}}\Vert P_{\leq N}\phi\Vert_{L^2(\mathcal{M})}.
\end{align}
\end{proposition} 
\begin{remark}The proof of Proposition \ref{epsilonremovalprop} simplifies the $\epsilon$-removal argument of Killip and Vişan \cite{KV14} in the case of the semi-periodic domain. The argument in \cite{KV14} is based on the splitting of the kernel $($associated with the linear propagator$)$ using major arcs
\begin{align}\label{majorarc}
\mathcal{T}:=\{t\in[0,1]: qN^2\big|t-a/q\big|\leq N^{2\sigma}\,\,\text{for some}\,\,q\leq N^{2\sigma},\,\text{and}\,(a,q)=1\},   
\end{align}and kernel estimate due to Weyl's method. In our situation, we introduce the short time intervals $I_N$ $($depending on dyadics $N)$ instead of the major arc $\mathcal{T}$; and complete the proof of Proposition \ref{epsilonremovalprop} by an elementary argument that replaces the Weyl sums estimate. In particular, taking advantage of dispersive decay separately in both periodic $($thanks to \cite[Theorem $5.3$]{KPV91}$)$ and the real line directions, we provide a simple proof of Strichartz type estimate for the kernel over $I_N$, whereas the proof for analogous estimate \cite[Proposition $2.3$]{KV14} is much more involved. See Subsection \ref{epsilonremoval} for the details.\end{remark}

We upgrade the local estimates \eqref{BDstrcest2} to a result of global Strichartz estimates by which we prove the global existence assertion of Theorem \ref{maintm}.
\begin{proposition}\label{removeglobalstric}
Let $p>4$ and $q=\frac{4p}{p-2}$. Then for all $N\geq 1$ and $\phi\in L^2(\mathbb{R}\times\mathbb{T})$, we have 
\begin{equation}
\begin{split}\label{Stricremove}
\Vert e^{it\Box} P_{\leq N}\phi\Vert_{\ell^qL^{p}(\mathbb{R},L^{p}(\mathbb{R}\times\mathbb{T}))}&\lesssim  N^{1-\frac{4}{p}}\Vert P_{\leq N}\phi\Vert_{L^2(\mathbb{R}\times\mathbb{T})}. 
\end{split}
\end{equation}
\end{proposition} 
Estimate \eqref{Stricremove} results from an $\epsilon$-removal argument after establishing the endpoint $\ell^8L^4$ estimate, which will be the main issue in Section \ref{GSE}, and our argument to prove such an estimate will be in the sense of Barron-Christ-Pausader \cite{BCP21}, where they consider the elliptic Schr\"odinger semi-group. This procedure generates sets \eqref{sets} that also appear in the proof of the local-in-time $L^4$ Strichartz estimate of Section \ref{LSE}. Hence, the result of Proposition \ref{removeglobalstric} depends on the bound for the size of such sets; see the proof of Theorem \ref{mainthmL4est}. We now introduce our notion of admissible pairs, which will be used in subsequent sections.       
\begin{definition}\label{admissibledefn}
  We say that $(p,q)\in \mathbb{R}\times \mathbb{R}$ is admissible pair if $\frac{2}{q}+\frac{1}{p}=\frac{1}{2}$ and $p>4$, equivalently, 
\begin{equation*}\label{admissible}
\begin{aligned}
 q=q(p):=\frac{4p}{p-2},\quad 4\leq q< 8.
\end{aligned}
\end{equation*}
\end{definition}
We remark that the admissible pairs $(p,q)$ defined above are also admissible pairs for the one-dimensional Schr\"odinger equation.

Next, let us return to the discussion pertaining to the propagator $e^{it\partial_x\partial_y}$ carried out at the beginning of this section. When the initial data is defined on $\mathbb{R}\times\mathbb{T}$ or $\mathbb{T}^2$, the inequality \eqref{strichartzeuclid} fails dramatically since the resonance of the free evolution is non-negligible. In this connection, we consider 
\begin{equation}\label{1.1}
(i\partial_t +\partial_x\partial_y)u=0,\quad x\in\R,\,y\in\T.
\end{equation}
 By considering solutions of type 
$$
u(t,x,y)=g(x),\, g\in H^s(\R),
$$
we obtain that for $s<1/4$, the estimate \begin{equation*}\label{2.1}
\|u(t)\|_{L^4([0,1]\times\R\times\T)}\leq C\|u(0)\|_{H^s(\R\times\T)}\,
\end{equation*} cannot hold for solutions of \eqref{1.1}. More precisely, this would be in a contradiction with the Sobolev embedding. Thus, one would expect the sharp Strichartz inequality: 
\begin{align}\label{strichforparxpary}
   \Vert e^{it\partial_x\partial_y}P_{\leq N}f\Vert_{L^4_{t,x,y}([0,1]\times \mathbb{R}\times\mathbb{T})}\lesssim N^{\frac{1}{4}}\Vert f\Vert_{L^2(\mathbb{R}\times\mathbb{T})},
 \end{align}
which is observed in the context of the Boltzmann equation, \cite[Theorem 2.1]{BSTW24}. Therefore, in view of estimates \eqref{L4witheps} and \eqref{strichforparxpary}, the linear part of \eqref{eq:HNLS} and the equation \eqref{1.1} behave very differently with respect to dispersive estimates. This is in contrast with the case $x\in\R$, $y\in\R$ when the linear HNLS in \eqref{eq:HNLS} and equation \eqref{1.1} are equivalent. We also point out that when the spatial domain is $\mathbb{T}^2$, the $L^4$ Strichartz estimate for the linear group $e^{it\Box}$ exhibits the same loss of $N^{\frac{1}{4}}$ as in \eqref{strichforparxpary}, see \cite{YW13_2, GT12}. We have the following slight improvement of \eqref{strichforparxpary}, which is of independent interest; for its proof, see Section \ref{LSEthm1.8}.
\begin{theorem}\label{meanzerothm}
Let $(x,y)\in\mathbb{R}\times\mathbb{T}$. Then, for $1\leq N$, we have
   \begin{align*}
\Vert e^{it\partial_x\partial_y}P_{\leq N}f\Vert_{L^4_{t,x,y}([0,1]\times \mathbb{R}\times
\mathbb{T})}\lesssim
(\log(N))^{\frac 1 4} \Vert f\Vert_{L^{2}(\mathbb{R}\times\mathbb{T})}
+
N^{\frac{1}{4}}\Vert \widehat{f}(\xi,n)\Vert_{
l^4_n L^2_\xi}.
\end{align*} 
\end{theorem}
The improvement comes from the fact that in the right hand-side of \eqref{strichforparxpary} we have 
$
\Vert \widehat{f}(\xi,n)\Vert_{
l^2_n L^2_\xi},
$
i.e., we have an improvement in the $n$ summability condition. 
\subsection{Outline of the paper} In Section \ref{preliminaries}, we introduce notation and function spaces. In Section \ref{LSE}, we prove non-scale invariant local-in-time Strichartz estimates, and by an $\epsilon$-removal argument, we then obtain scale invariant versions of these Strichartz estimates. In Section \ref{GSE}, we upgrade the Strichartz estimates of Section \ref{LSE} to the global Strichartz estimates. Lastly, in Section \ref{SectNonlinear}, we prove Theorem \ref{maintm} by using the results in Section \ref{GSE}.
\section{Preliminaries}\label{preliminaries}
\subsection{Notation}
 We write $A\lesssim B$ to indicate that there is a constant $C>0$ such that $A\leq CB$, also denote $A\sim B$ when $A\lesssim B \lesssim A$. Given $\epsilon>0$, the notation $C_{\epsilon}$ is used for a constant depending on $\epsilon>0$, and we also define $\lesssim_{\epsilon}$ accordingly. We define the Fourier transform of functions as follows
\begin{equation*}
 \begin{aligned}
     \mathcal{F}_{x,y}[f](\xi,k)&=\widehat{f}(\xi,k)=\int_{\mathbb{R}\times\mathbb{T}}e^{-2\pi i(\xi x+ky)}f(x,y)dxdy,  \\ \mathcal{F}_{\xi,k}^{-1}[\widehat{f}](x,y)&=f(x,y)=\sum_{k\in\mathbb{Z}}\int_{\mathbb{R}}e^{2\pi i(\xi x+ky)}\widehat{f}(\xi,k)d\xi.
 \end{aligned}
 \end{equation*}
 The partial Fourier transforms are also defined accordingly. In Subsection \ref{subsectionmainthm}, as the analysis will be performed on the frequency side, for the sake of dropping the factor of $2\pi$, we prefer to use the other convention for the Fourier transform. Let $\varphi$ denote a smooth radial cutoff such that $\varphi(x)=1$ for $|x|\leq 1$ and $\varphi(x)=0$ for $|x|\geq 2$. For a dyadic number $N\geq 1$, we define the Littlewood-Paley projection operators $P_{\leq N}$ and $P_N$ as follows:
 \begin{equation*}
\begin{aligned}
\widehat{P_{\leq N}f}(\xi,k)&:=\varphi(\xi/N)\varphi(k/N)\widehat{f}(\xi,k),\quad (\xi,k)\in\mathbb{R}\times\mathbb{Z},
\end{aligned}
\end{equation*}
and
\begin{equation*}
\begin{aligned}
P_1f&:=P_{\leq1}f,\\
P_Nf&:=P_{\leq N}f-P_{\leq \frac{N}{2}}f,\quad\text{if}\,\,N\geq 2.
\end{aligned}
\end{equation*}
More generally, for any measurable set $S\subset \mathbb{R}\times\mathbb{Z}$, we write $P_S$ to denote the Fourier projection operator with symbol $\mathds{1}_{S}$, where $\mathds{1}_{S}$ denotes the characteristic function of $S$. Lastly, define the norm on $\mathbb{R}\times\mathcal{M}=(\mathbb{Z}+[0,1))\times\mathcal{M}$:
\begin{equation*}\label{MixedTypeNorms}
\Vert u\Vert^q_{\ell^qL^p(\mathbb{R},L^r(\mathcal{M}))}:=\sum_{\gamma\in\mathbb{Z}}\left(\int_{s\in[0,1)}\left( \int_{\mathcal{M}}\vert u(\gamma+s,\mathbf{x})\vert^r\, d\mathbf{x}\right)^\frac{p}{r}ds\right)^\frac{q}{p}.
\end{equation*}
If $p=r$, we also use the notation $\Vert u\Vert_{\ell^q_{\gamma}L^p_{t,\mathbf{x}}([\gamma,\gamma+1]\times\mathcal{M})}$ in substitution for $\Vert u\Vert_{\ell^qL^p(\mathbb{R},L^p(\mathcal{M}))}$.
\subsection{Hardy-Littlewood-Sobolev inequality}
We record the discrete analogue of Hardy-Littlewood-Sobolev inequality, which is to be used in due course. 
\begin{lemma}\label{HLSineq}
For any $1<p, r<\infty$ and $0<\alpha<1$, we have
\begin{align*}
    \sum_{j\neq k}\frac{a_jb_k}{|j-k|^{\alpha}}\lesssim \Vert a\Vert_{\ell^p}\Vert b \Vert_{\ell^r},\quad \frac{1}{p}+\frac{1}{r}+\alpha=2.
\end{align*}
\end{lemma}
\subsection{Function spaces}
In this subsection, we introduce the main function spaces $X^s$ and $Y^s$ in the context of HNLS, where the related definitions are as in \cite{HHK09, HTT11, HTT14}. These types of spaces are initially used as a substitute for Fourier restriction spaces to construct solutions at the critical regularity for dispersive PDEs, see \cite{HHK09, HTT11, HTT14}. Let $\mathcal{H}$ be a separable Hilbert space over $\mathbb{C}$, and we write $\mathcal{Z}$ to denote the set of finite partitions $-\infty<t_0<t_1<...<t_K\leq \infty$ of the real line with the convention that $v(\infty):=0$ for all functions $v: \mathbb{R}\rightarrow \mathcal{H}$.  
 \begin{definition}
Given $1\leq p <\infty$, a $U^p$-atom is a piecewise defined function $a:\mathbb{R}\rightarrow \mathcal{H}$ of the form
 \begin{align*}
     a=\sum_{k=1}^K\mathds{1}_{[t_{k-1},t_k)}\phi_{k-1}
 \end{align*}
  where $\{t_k\}_{k=0}^K\in \mathcal{Z}$ and $\{\phi_k\}_{k=0}^{K-1}\subset \mathcal{H}$ with $\sum_{k=0}^{K-1}\Vert \phi_k\Vert_{\mathcal{H}}^p=1$. The atomic space $U^p(\mathbb{R};\mathcal{H})$ is defined to be the set of all functions $u:\mathbb{R}\rightarrow \mathcal{H}$ such that
 \begin{align*}
  u=\sum_{j=1}^{\infty}\lambda_ja_j\quad \text{for}\,\,U^p-\text{atoms} \,\,a_j,\,\,\{\lambda_j\}\in \ell^1(\mathbb{C}),
 \end{align*}
with norm
 \begin{align*}
     \Vert u\Vert_{U^p}:=\inf\Big\{\sum_{j=1}^{\infty}|\lambda_j|: u=\sum_{j=1}^{\infty}\lambda_ja_j\,\,\text{with}\,\,\{\lambda_j\}\in \ell^1(\mathbb{C})\,\,\text{and}\,\,U^p-\text{atoms}\,\,a_j\Big\}. \end{align*}
 \end{definition}
 \begin{definition}
 Let $1\leq p<\infty$. \begin{itemize}
 \item[(i)] We define $V^p(\mathbb{R};\mathcal{H})$ as the space of all functions $v: \mathbb{R}\rightarrow \mathcal{H}$ such that 
 \begin{align}\label{V^pnorm}
     \Vert v\Vert_{V^p}:=\sup_{\{t_k\}_{k=0}^K\in\mathcal{Z}}\Big(\sum_{k=1}^K\Vert v(t_k)-v(t_{k-1})\Vert_{\mathcal{H}}^p\Big)^{\frac{1}{p}}<\infty.
 \end{align}
  \item[(ii)] The space $V^p_{rc}(\mathbb{R};\mathcal{H})$ denotes the closed subspace of all right-continuous $v\in V^p$ such that $\lim_{t\rightarrow-\infty}v(t)=0$, endowed with the same norm \eqref{V^pnorm}.
 \end{itemize}
 \end{definition}
 \begin{definition}
For $s\in\mathbb{R}$, we define $U_{\mathcal{\Box}}^pH^s$ and $V_{\mathcal{\Box}}^pH^s$ as the spaces of all functions $u: \mathbb{R}\rightarrow H^s(\mathbb{R}\times\mathbb{T})$ such that the map $t \rightarrow e^{-it\mathcal{\Box}} u(t)$ is in $U^p(\mathbb{R};H^s)$ and $V^p_{rc}(\mathbb{R};H^s)$ respectively, with norms 
\begin{align*}
    \Vert u\Vert_{U^p_{\Box}H^s}:= \Vert e^{-it\mathcal{\Box}}u\Vert_{U^p(\mathbb{R};H^s)},\quad \Vert u\Vert_{V^p_{\Box}H^s}:= \Vert e^{-it\mathcal{\Box}}u\Vert_{V^p(\mathbb{R};H^s)}.
\end{align*}
 \end{definition}
 For $z\in\mathbb{Z}^2$, define the cube $C_z:=z+[0,1)^2$, by which we introduce a disjoint partition $\bigcup_{z\in\mathbb{Z}^2}C_z=\mathbb{R}^2$, and let $P_{C_z}$ denote the Fourier projection operator on the cube $C_z$ (with symbol $\mathds{1}_{C_z}$). Thus, the following spaces will be central to our discussion of well-posedness.
 \begin{definition}
 Let $s\in\mathbb{R}$. 
 \begin{itemize}
\item[(i)] We define $X^s$ as the space of all functions $u: \mathbb{R}\rightarrow H^s(\mathbb{R}\times\mathbb{T})$ such that for every $z\in \mathbb{Z}^2$, $P_{C_z}u\in U^2_{\Box}(\mathbb{R}; H^s(\mathbb{R}\times\mathbb{T}))$, and the norm \begin{equation*}
 \begin{aligned}
 \Vert u\Vert_{X^s}&:=\Big(\sum_{z\in\mathbb{Z}^2}\Vert P_{C_z}u\Vert_{U^2_{\Box}(\mathbb{R}; H^s)}^2\Big)^{\frac{1}{2}}
 \end{aligned}
 \end{equation*}
is finite.
\item[(ii)] We define $Y^s$ as the space of all functions $u: \mathbb{R}\rightarrow H^s(\mathbb{R}\times\mathbb{T})$ such that for every $z\in \mathbb{Z}^2$, $P_{C_z}u\in V^2_{\Box}(\mathbb{R}; H^s(\mathbb{R}\times\mathbb{T}))$, and the norm \begin{equation*}
 \begin{aligned}
 \Vert u\Vert_{Y^s}&:=\Big(\sum_{z\in\mathbb{Z}^2}\Vert P_{C_z}u\Vert_{V^2_{\Box}(\mathbb{R}; H^s)}^2\Big)^{\frac{1}{2}}
 \end{aligned}
 \end{equation*}
 is finite.
 \end{itemize}
 \end{definition}
 \begin{remark}\label{remark2.6}
 For $p>2$, we have the following continuous embeddings
 \begin{align*}\label{U^2trXsembed}
  U^2_{\Box}H^s\hookrightarrow X^s \hookrightarrow Y^s \hookrightarrow V^2_{\Box}H^s\hookrightarrow U^p_{\Box}H^s\hookrightarrow L^{\infty}(\mathbb{R};H^s).  
 \end{align*}
 \end{remark}
For a bounded interval $I\subset \mathbb{R}$, the time-restriction norms for $U^p(I), V^p(I), X^s(I)$ and $Y^s(I)$ can be defined in the natural way. We also have the following lemma, which is helpful in constructing the solutions to \eqref{eq:HNLS}.  
\begin{lemma}\label{linearestlemma}
 Let $s\geq 0$, $I\subset \mathbb{R}_{\geq 0}$, $f\in H^s(\mathbb{R}\times\mathbb{T})$, and $F\in L^1(I;H^{s}(\mathbb{R}\times\mathbb{T}))$. Then, for all $0\leq t_0<t\in I$, we have
 \begin{equation*}
 \begin{aligned}
  \Vert e^{it\Box}f\Vert_{Y^s(I)}&\leq \Vert e^{it\Box}f\Vert_{X^s(I)} \leq \Vert f\Vert_{H^s(\mathbb{R}\times\mathbb{T})}, \\
 \Vert u \Vert_{L^{\infty}_tH^s_x(I\times \mathbb{R}\times\mathbb{T})}&\lesssim  \Vert u\Vert_{Y^s(I)}\lesssim  \Vert u\Vert_{X^s(I)},
\\ \Big\Vert \int_{t_0}^te^{i(t-\tau)\Box} F(\tau)d\tau\Big\Vert_{X^s(I)} &\lesssim \Vert F\Vert_{L^1_tH^s_{x,y}(I\times\mathbb{R}\times\mathbb{T})}.
 \end{aligned}
 \end{equation*}
 \end{lemma}

\section{Local Strichartz estimates}\label{LSE}
In this section, our main goal is to establish the scale-invariant Strichartz estimates of Proposition \ref{epsilonremovalprop}. In this regard, we first prove the following theorem:
\begin{theorem}\label{mainthmL4est}
 For all $N\geq 1$ and any $\epsilon>0$, we have
\begin{align}\label{L4witheps}
\Vert e^{it\Box_{\mathbf{x}}}P_{\leq N}\phi\Vert_{L^4_{t,\mathbf{x}}([0,1]\times \mathcal{M})}\lesssim_{\epsilon} N^{\epsilon}\Vert P_{\leq N}\phi\Vert_{L^2(\mathcal{M})}.
\end{align}
\end{theorem}
 In \cite{HST25}, there is an alternative proof of the $L^4$ Strichartz estimate \eqref{L4witheps} based on decoupling estimates due to Guth-Maldague-Oh \cite{GMO24}. Moreover, it is worth noting that after our paper was published, Deng-Fan-Zhao \cite{DFZ25} removed the $\epsilon$-loss in the estimate \eqref{L4witheps} by a proof based on the combination of a kernel decomposition method with measure estimates for semi-algebraic sets. Hence, using the sharp Strichartz estimate, they established the global well-posedness for the cubic HNLS \eqref{eq:HNLS} in $L^2$ with small data, which is left from Theorem \ref{maintm}. Note that interpolating the estimate \eqref{L4witheps} and the following straightforward estimate
 \begin{align*}
    \Vert e^{it\Box_{\mathbf{x}}}P_{ \leq N}\phi\Vert_{L^{\infty}_{t,\mathbf{x}}([0,1]\times\mathcal{M})}\lesssim N \Vert \phi\Vert_{L^2_{\mathbf{x}}(\mathcal{M})}
\end{align*}
yields the following non-scale invariant estimates  
\begin{corollary}\label{maincorofthm}
 For all $N\geq 1$, $p\geq 4$, and any $\epsilon>0$, we have
\begin{align*}\label{Lpwitheps}
\Vert e^{it\Box_{\mathbf{x}}}P_{\leq N}\phi\Vert_{L^p_{t,x}([0,1]\times \mathcal{M})}\lesssim_{\epsilon} N^{1-\frac{4}{p}+\epsilon}\Vert P_{\leq N}\phi\Vert_{L^2(\mathcal{M})}.
\end{align*}
\end{corollary}
\begin{proof}[Proof of Theorem \ref{mainthmL4est}]
After some classical reductions  (see e.g. \cite{HTT14}, page~75 or \cite{TT}) based on the Littlewood-Paley theory and the use of the Cauchy-Schwarz inequality, the issue is to bound 
$$
{\rm mes}(A_{N}(\tau,\xi,n)),
$$ 
where ${\rm mes}$ refers to the product of the Lebesgue measure on $\R$ and the counting measure on $\Z$ and $A_{N}(\tau,\xi, n)$ is defined by 
\begin{multline*}
A_N(\tau, \xi, n)=
\big\{
(\xi_1,n_1)\in \R\times\Z\,:\, 
|\tau+\xi_1^2+(\xi-\xi_1)^2-n_1^2-(n-n_1)^2|\leq C,\\
 |n_1|\leq 2N,\,\,\, |n-n_1|\leq 2N
\big\},
\end{multline*}
where $\tau, \xi\in \R$, $n\in \Z$ and $N$ is a dyadic integer.  We can write
$$
n_1^2+(n-n_1)^2=2\big(n_1-\frac{n}{2}\big)^2+\frac{n^2}{2}
$$
and similarly for the expression involving $\xi_1,\xi$. Using the invariance of the Lebesgue measure under translations, we obtain that the issue boils down to bound 
\begin{equation}\label{3}
{\rm mes}(B_{N}(\tau,n)),
\end{equation}
where $B_{N}(\tau,n)$ is defined by 
\begin{align*}
B_N(\tau, n)=
\big\{
(\xi_1,n_1)\in \R\times\Z\,:\, 
|\tau+\xi_1^2-(n-n_1)^2|\leq C,\, |n_1|\leq 2N,\,\,\, |n-n_1|\leq 2N
\big\},
\end{align*}
where $\tau, n\in \R$ and $N$ is a dyadic integer.  
Set 
$$
A=\tau-(n-n_1)^2\,.
$$
We know that
$$
-C\leq A+\xi_1^2\leq C
$$
and as a consequence $A\leq C$. If in addition $A\geq -100 C$ then the number of possible $n_1$ is $O(1)$.
Moreover, in the case $A\geq -100 C$, we have
$$
\xi_1^2\leq -A+C\leq 101 C
$$
and therefore the measure of the possible $\xi_1$ is $O(1)$. Consequently when $A\geq -100 C$  we have that \eqref{3} is  $O(1)$.

Therefore, from now on we can suppose that $A\leq -100 C$. For a fixed $n_1$, the measure of the possible $\xi_1$ is
$$
\sqrt{C-A}-\sqrt{-C-A}=\frac{2C}{\sqrt{C-A}+\sqrt{-C-A}}\lesssim \frac{1}{\sqrt{-A}}\,.
$$
Therefore, we can write
\begin{equation*}
{\rm mes}(B_{N}(\tau,n))
\lesssim
1+\sum_{\substack{|n_1|\leq 2N \\  |n-n_1|\leq 2N \\ (n-n_1)^2-\tau\geq 100C}}
\frac{1}{\big( (n-n_1)^2-\tau\big)^{\frac{1}{2}}}\,.
\end{equation*}
We now claim that the last sum is $\lesssim \log(N)$. 
Indeed, the contribution of 
$$
 (n-n_1)^2-\tau\geq N^2
 $$
 is bounded by $O(1)$. 
 Next, for $L\geq 1$, we have that
 $$
 \#\{n_1\in\Z\,:\,
 (n-n_1)^2-\tau\in [L,2L]
\}
\lesssim \sqrt{L}\,.
$$
Therefore the contribution of the region 
$$
 (n-n_1)^2-\tau\leq N^2
 $$
 can be bounded by
 $$
 \sum_{\substack{
 L-{\rm dyadic}
 \\
 L\lesssim N^2
 }}
 \frac{\sqrt{L}}{\sqrt{L}}\lesssim \log(N)\,.
 $$
The assertion follows and the proof is complete.
\end{proof}
\subsection{The $\epsilon$-Removal Argument}\label{epsilonremoval}
In this subsection, we will prove Proposition \ref{epsilonremovalprop}. Note that by a hyperbolic-type Galilean transformation as in \cite{YW13_2}, Proposition \ref{epsilonremovalprop} leads to the corollary in the sequel. First, we denote by $\mathscr{C}_N$ the collection of cubes $C\subset \mathbb{R}\times\mathbb{Z}$ of side length $N\geq 1$ with arbitrary center and orientation.
\begin{corollary}\label{stricharrtzcor}
Let $0< T \leq 1$. For all $C\in \mathscr{C}_N$ and $p>4$, we have
\begin{equation*}\label{StricartzwithC}
\begin{aligned}
  \Vert P_Ce^{it\Box}\phi\Vert_{L^p([0,T]\times \mathcal{M})} \lesssim N^{1-\frac{4}{p}}\Vert P_C\phi\Vert_{L^2(\mathcal{M})}. 
\end{aligned}
\end{equation*}
\end{corollary}
For $\mathbf{x}=(x,y)$, we shall denote the truncated convolution kernel corresponding to the hyperbolic propagator as follows
\begin{equation}\label{kernel}
\begin{aligned}
 K_N(t,\mathbf{x})&:=[e^{it\Box}P_{\leq N}\delta_0](\mathbf{x})\\&=\Bigg[\int_{\mathbb{R}}\varphi(\xi/N)e^{2\pi i(\xi x-t\xi^2)}d\xi\Bigg]\cdot\Bigg[\sum_{k\in\mathbb{Z}}\varphi(k/N)e^{2\pi i(ky+tk^2)}\Bigg]\\&=:K_{NI}(t,x)\otimes K_{NS}(t,y).  
 \end{aligned}
 \end{equation}
\begin{proof}[Proof of Proposition \ref{epsilonremovalprop}]
Let $\Vert f\Vert_{L^2(\mathcal{M})}=1$. Thus via Bernstein's inequality, we have
\begin{align*}
    \Vert e^{it\Box}P_{\leq N}f\Vert_{L^{\infty}_{t,\mathbf{x}}([0,1]\times\mathcal{M})}\leq CN
\end{align*}
for some $C>0$. Fix $p>4$, then using the inequality above we may write
\begin{align*}
 \Vert e^{it\Box}P_{\leq N}f\Vert_{L^{p}_{t,\mathbf{x}}([0,1]\times \mathcal{M})}^p=\int_0^{CN}p\mu^{p-1}\big|\{(t,\mathbf{x})\in [0,1]\times\mathcal{M}:|(e^{it\Box}P_{\leq N}f)(\mathbf{x})|>\mu\}\big|d\mu.\end{align*}
For small $\delta>0$ to be determined later and $\epsilon\leq \frac{\delta(p-4)}{4}$, we apply Chebyshev's inequality and Theorem \ref{mainthmL4est} to obtain
 \begin{equation}\label{firstest}
\begin{aligned}
 \int_0^{N^{1-\delta}}&p\mu^{p-1}\big|\{(t,\mathbf{x})\in [0,1]\times\mathcal{M}:|(e^{it\Box}P_{\leq N}f)(\mathbf{x})|>\mu\}\big|d\mu \\&\lesssim N^{4\epsilon}\int_0^{N^{1-\delta}}\mu^{p-5}d\mu\\&\lesssim N^{p(1-\frac{4}{p})+4\epsilon -\delta(p-4)}\lesssim N^{p(1-\frac{4}{p})}.  
\end{aligned}
\end{equation}
Next, for fixed $\mu>N^{1-\delta}$, let us denote 
\begin{align*}
    \Omega=\{(t,\mathbf{x})\in [0,1]\times\mathcal{M}:|(e^{it\Box}P_{\leq N}f)(\mathbf{x})|>\mu\}
\end{align*}
and write
\begin{align*}
I:= \int_{N^{1-\delta}}^{CN}p\mu^{p-1}\big|\Omega\big|d\mu. 
\end{align*}
To control $I$, we introduce the set \begin{align*}
    \Omega_{\omega}=\{(t,\mathbf{x})\in [0,1]\times\mathcal{M}:\text{Re}(e^{i\omega}e^{it\Box}P_{\leq N}f)(\mathbf{x})>\frac{\mu}{2}\}
\end{align*}
that satisfies $|\Omega|\leq 4|\Omega_{\omega}|$ for some $\omega\in \{0, \frac{\pi}{2}, \pi, \frac{3\pi}{2}\}$. So we need to estimate $\Omega_{\omega}$. For this purpose, employing Cauchy-Schwarz inequality, it follows that
\begin{equation}\label{Omega_omegaest}
\begin{aligned}
\mu^2|\Omega_{\omega}|^2&\lesssim \langle e^{it\Box}P_{\leq N}f, \mathds{1}_{\Omega_{\omega}} \rangle_{L^2_{t,\mathbf{x}}}^2\\&=\langle f, e^{-it\Box}P_{\leq N}\mathds{1}_{\Omega_{\omega}} \rangle_{L^2_{t,\mathbf{x}}}^2\\&\lesssim \langle \mathds{1}_{\Omega_{\omega}}, K_N*\mathds{1}_{\Omega_{\omega}}\rangle_{L^2_{t,\mathbf{x}}}.
\end{aligned}
\end{equation}
The following bound for the kernel $K_{N}$ will be of fundamental importance in controlling the resulting bound in \eqref{Omega_omegaest}.
\begin{lemma}\label{kerestlem}
For $N\geq 1$, $\mathbf{x}\in\mathcal{M}$, we have 
\begin{align}\label{kernelest}
    |K_N(t,\mathbf{x})|\lesssim\begin{cases}N|t|^{-\frac{1}{2}},\quad \text{for all}\,\,t\in\mathbb{R},\\
     |t|^{-1},\qquad \text{for}\,\,|t|\leq N^{-1}.
    \end{cases}
    \end{align}
\end{lemma}
\begin{proof}
It is well-known that 
\begin{align}\label{dispestR}
|K_{NI}(t,x)|\lesssim |t|^{-\frac{1}{2}},\quad \forall t\in \mathbb{R}.    
\end{align} 
 Thus, by \eqref{kernel}, the estimate \eqref{dispestR} and the straightforward bound $| K_{NS}|\lesssim N$ give rise to the first estimate in \eqref{kernelest} which holds for all times. Then, due to the estimate $(5.9)$ in the proof of \cite[Theorem $5.3$]{KPV91}, we also get the bound 
\begin{align}\label{dispestT}
   |K_{NS}(t,y)|\lesssim |t|^{-\frac{1}{2}},\quad \text{for}\,\,|t|\leq N^{-1}.
\end{align}
 Hence, by \eqref{kernel}, \eqref{dispestR} and \eqref{dispestT} imply the second estimate in \eqref{kernelest}.
\end{proof}
\begin{remark}
The analogue of the estimate $(5.9)$ in \cite{KPV91} or \eqref{dispestT} also holds for any compact Riemannian manifold without boundary; see \cite[Lemma $2.5$]{BGT04}. Intuitively, the wave packets at frequency scale $N$ travel at group velocity $\sim N$. Therefore, within the time scale $O(N^{-1})$, there is no formation of caustics in the WKB approximation, which is responsible for the same dispersive estimate as the case of $\mathbb{R}^d$.   
\end{remark}
In view of Lemma \ref{kerestlem}, we set the short-time interval
\begin{align*}
  I_N:=[0,N^{-1}],\quad\text{for}\,\, N\geq 1,  
\end{align*} 
and define 
\begin{align*}
K_{N,0}(t,\mathbf{x}):=\mathds{1}_{I_N}(t)K_N(t,\mathbf{x}).
\end{align*} Hence using \eqref{kernelest}, we have
\begin{align*}
    |K_N(t,\mathbf{x})-K_{N,0}(t,\mathbf{x})|\lesssim Nt^{-\frac{1}{2}}\leq N^{\frac{3}{2}},\quad \forall t\in [0,1]\setminus I_N.
\end{align*}
This, together with Hölder's inequality and Young's inequality, gives rise to the bound
\begin{equation}\label{K_N-K_Ntilde}
\begin{aligned}
 |\langle \mathds{1}_{\Omega_{\omega}},(K_N-K_{N,0})*\mathds{1}_{\Omega_{\omega}}\rangle_{L^2_{t,\mathbf{x}}}|&\lesssim \Vert \mathds{1}_{\Omega_{\omega}}\Vert_{L^1_{t,\mathbf{x}}}\Vert (K_N-K_{N,0})*\mathds{1}_{\Omega_{\omega}}\Vert_{L^{\infty}_{t,\mathbf{x}}} \\&\lesssim |\Omega_{\omega}|^2N^{\frac{3}{2}}. 
\end{aligned}
\end{equation}
By Lemma \ref{kerestlem}, we have
$$ |K_{N,0}(t,\mathbf{x})|\lesssim \epsilon_N(t):=\mathds{1}_{I_N}(t)\min\{Nt^{-\frac{1}{2}},t^{-1}\}.
$$
Notice that, when $0<t<N^{-2}$, we have $\epsilon_N(t)=Nt^{-\frac{1}{2}}$, and for $N^{-2}\leq t\leq N^{-1}$, we have $\epsilon_N(t)=t^{-1}$. It is worth noting that though \cite[Proposition $2.3$]{KV14} is applicable in our analysis in the rest (when approximating kernel is defined over $\mathcal{T}$, see \eqref{majorarc}), here, we present a straightforward proof thanks to the decay in Lemma \ref{kerestlem}.
\begin{lemma}\label{kernlestlemma}
Fix $r\in (4,6)$, then we have
$$ \|K_{N,0}\ast F\|_{L_{t,\mathbf{x}}^r([0,1]\times\mathcal{M} )}\lesssim N^{2-\frac{8}{r}}\|F\|_{L_{t,\mathbf{x}}^{r'}([0,1]\times\mathcal{M})}.
$$
\end{lemma}
\begin{proof}
We follow the standard argument as performed for the proof of the Strichartz inequality in $\mathbb{R}^d$. Note that for a space-time function $F$, 
$$ K_{N,0}\ast F(t)=\int_{0}^1 K_{N,0}(t-s)\ast_{\mathbf{x}} F(s) ds,
$$
where $\ast_{\mathbf{x}}$ means the convolution for the spatial variable $\mathbf{x}\in\mathcal{M}$.
 From the dispersive estimate, we have
 $$ \|K_{N,0}(t)\ast_{\mathbf{x}}\|_{L_{\mathbf{x}}^1\rightarrow L_{\mathbf{x}}^{\infty}}\lesssim \epsilon_N(t).
 $$
 Since $K_{N,0}(t)$ is the kernel of the truncated propagator $\mathds{1}_{I_N}(t)e^{it\Box}P_{\leq N}$, we have
 $$ \|K_{N,0}(t)\ast_{\mathbf{x}}\|_{L_{\mathbf{x}}^2\rightarrow L_{\mathbf{x}}^2}\leq 1.
 $$
 Then, Riesz interpolation gives
 $$ \|K_{N,0}(t)\ast_{\mathbf{x}}\|_{L_{\mathbf{x}}^{r'}\rightarrow L_{\mathbf{x}}^r }\lesssim \epsilon_N(t)^{1-\frac{2}{r}},\quad \forall 2<r<\infty.
 $$
Plugin, 
$$ \|K_{N,0}\ast F\|_{L_{t,\mathbf{x}}^r([0,1]\times\mathcal{M})}\lesssim \Big\|
\int_0^1  \big(\epsilon_N(t-s)\big)^{1-\frac{2}{r}} \|F(s)\|_{L_{\mathbf{x}}^{r'}} ds
\Big\|_{L_t^r([0,1])}.
$$
By Schur test, the above inequality is bounded by
$$ A_N\|F\|_{L_{t,\mathbf{x}}^{r'}},
$$
with
$$ A_N=\sup_{s\in[0,1]}\|\epsilon_N(t-s)\|_{L_t^{\frac{r}{2}-1}([0,1])}^{1-\frac{2}{r}}+\sup_{t\in[0,1]}\|\epsilon_N(t-s)\|_{L_t^{\frac{r}{2}-1}([0,1])}^{1-\frac{2}{r}}\leq 2\|\epsilon_N(t)\|_{L_t^{\frac{r}{2}-1}(I_N)}^{1-\frac{2}{r}}.
$$
Thus, for $r\in (4,6)$, direct computation yields
\begin{align*}
\|\epsilon_N(t)\|_{L_t^{\frac{r}{2}-1}(I_N)}^{\frac{r}{2}-1}\leq \int_0^{N^{-2}}(Nt^{-\frac{1}{2}})^{\frac{r}{2}-1} dt +\Big|\int_{N^{-2}}^{N^{-1}}t^{-(\frac{r}{2}-1)}dt\Big|
\lesssim N^{r-4}.
\end{align*}
Hence $A_N\lesssim N^{2-\frac{8}{r}}$. 

\end{proof}
Fix $4<r<\min\{6,p\}$. As a consequence of Lemma \ref{kernlestlemma}, we get
\begin{equation}\label{K_Ntilde}
\begin{aligned}
 |\langle \mathds{1}_{\Omega_{\omega}},K_{N,0}*\mathds{1}_{\Omega_{\omega}}\rangle_{L^2_{t,\mathbf{x}}}|&\lesssim \Vert \mathds{1}_{\Omega_{\omega}}\Vert_{L^{r'}_{t,\mathbf{x}}}\Vert K_{N,0}*\mathds{1}_{\Omega_{\omega}}\Vert_{L^{r}_{t,\mathbf{x}}} \\&\lesssim |\Omega_{\omega}|^{\frac{2}{r'}}N^{2-\frac{8}{r}}. 
\end{aligned}
\end{equation}
We choose $\delta\ll \frac{1}{4}$. Then, substituting the estimates \eqref{K_N-K_Ntilde} and \eqref{K_Ntilde} into \eqref{Omega_omegaest} entails that
\begin{align*}
  |\Omega_{\omega}|&\lesssim N^{\frac{r}{2}(2-\frac{8}{r})}\mu^{-r}
\end{align*}
where we have used the fact that the contribution of \eqref{K_N-K_Ntilde} is much smaller than $\mu^2|\Omega_{\omega}|^2$ under assumptions $\mu>N^{1-\delta}$ and $\delta\ll \frac{1}{4}$. As a consequence, for $4<r<\min\{6,p\}$, we obtain 
\begin{equation*}
\begin{aligned}
I&\leq 4\int_{N^{1-\delta}}^{CN}p\mu^{p-1}\big|\Omega_{\omega}\big|d\mu&\lesssim N^{\frac{r}{2}(2-\frac{8}{r})}\int_{N^{1-\delta}}^{CN}\mu^{p-r-1}d\mu\lesssim N^{p(1-\frac{4}{p})}.
\end{aligned}
\end{equation*}
Thus, from the above estimate and \eqref{firstest}, \eqref{BDstrcest2} follows.
\end{proof}
\subsection{Proof of Theorem \ref{meanzerothm}}\label{LSEthm1.8}
Set
$
u(t,x,y)= e^{it\partial_x\partial_y}P_{\leq N}f\,.
$
We have 
\begin{multline*}
u^2(t,x,y)=\sum_{n_1, n_2\in\mathbb{Z}}\,\,\int_{\mathbb{R}^2}e^{2\pi i((\xi_1 n_1+\xi_2 n_2)t+
(\xi_1+\xi_2) x+(n_1+n_2)y)}
\\
\varphi(\xi_1/N)\varphi(n_1/N)\,\varphi(\xi_2/N)\varphi(n_2/N)\,\widehat{f}(\xi_1,n_1)\widehat{f}(\xi_2,n_2)d\xi_1 d\xi_2 .
\end{multline*}
Write 
$$
u^2(t,x,y)=u_1(t,x,y)+u_2(t,x,y),
$$
where 
\begin{multline*}
u_1(t,x,y)=\sum_{n_1\neq n_2}\,\,\int_{\mathbb{R}^2}e^{2\pi i((\xi_1 n_1+\xi_2 n_2)t+
(\xi_1+\xi_2) x+(n_1+n_2)y)}
\\
\varphi(\xi_1/N)\varphi(n_1/N)\,\varphi(\xi_2/N)\varphi(n_2/N)\,\widehat{f}(\xi_1,n_1)\widehat{f}(\xi_2,n_2)d\xi_1 d\xi_2 
\end{multline*}
and
\begin{equation*}
u_2(t,x,y)=\sum_{n\in\mathbb{Z}}\,\,\int_{\mathbb{R}^2}e^{2\pi i((\xi_1+\xi_2) (x+nt)+2ny)}
\varphi(\xi_1/N)\,\varphi(\xi_2/N)\varphi^2(n/N)\,\widehat{f}(\xi_1,n)\widehat{f}(\xi_2,n)d\xi_1 d\xi_2.
\end{equation*}
We have that
\begin{multline*}
\|u_2(t,x,\cdot)\|^2_{L^2(\T)}=
\\
\sum_{n\in\mathbb{Z}}\,
\Big|
\int_{\mathbb{R}^2}e^{2\pi i((\xi_1+\xi_2) (x+nt))}
\varphi(\xi_1/N)\,\varphi(\xi_2/N)\varphi^2(n/N)\,\widehat{f}(\xi_1,n)\widehat{f}(\xi_2,n)d\xi_1 d\xi_2
\Big|^2\,.
\end{multline*}
Set
$$
F_n(t,x)=\int_{\mathbb{R}^2}e^{2\pi i((\xi_1+\xi_2) (x+nt))}\varphi(\xi_1/N)\,\varphi(\xi_2/N)\varphi^2(n/N)\,\widehat{f}(\xi_1,n)\widehat{f}(\xi_2,n)d\xi_1 d\xi_2.
$$
We need to evaluate  $\| F_n\|_{L^2([0,1]\times\R)}.$ Write
$$
F_n(t,x)=\int_{\mathbb{R}}
e^{2\pi i\xi (x+nt)}
\Big(
\int_{\mathbb{R}}
\varphi(\xi_1/N)\,\varphi((\xi-\xi_1)/N)\varphi^2(n/N)\,\widehat{f}(\xi_1,n)\widehat{f}(\xi-\xi_1,n)d\xi_1 
\Big)
d\xi.
$$
Using the variable change $x'=x+nt$ (the $t$ integration disappears), we can write 
$$
\| F_n\|_{L^2([0,1]\times\R)}^2
=
\int_{\mathbb{R}}
\Big|
\int_{\mathbb{R}}
\varphi(\xi_1/N)\,\varphi((\xi-\xi_1)/N)\varphi^2(n/N)\,\widehat{f}(\xi_1,n)\widehat{f}(\xi-\xi_1,n)d\xi_1 
\Big|^2
d\xi.
$$
Using Cauchy-Schwarz  in the $\xi_1$ integration we get 
$$
\| F_n\|_{L^2([0,1]\times\R)}^2
\lesssim N \|\widehat{f}(\xi,n)\|_{L^2_\xi}^4
$$
which in turn implies 
$$
\Vert u_2\Vert^2_{L^2_{t,x,y}([0,1]\times \mathbb{R}\times
\mathbb{T})}
\lesssim
N\Vert \widehat{f}(\xi,n)\Vert^4_{
l^4_n L^2_\xi}.
$$
Let us next estimate 
$
\Vert u_1\Vert_{L^2_{t,x,y}([0,1]\times \mathbb{R}\times
\mathbb{T})}.
$
After some classical reductions  based on the Littlewood-Paley theory and the use of the Cauchy-Schwarz inequality, it suffices to bound 
$$
{\rm mes}(A_{N}(\tau,\xi,n)),
$$ 
 where
\begin{multline*}
A_N(\tau, \xi, n)=
\big\{
(\xi_1,n_1)\in \R\times\Z\,:\, 
|\tau+
%\xi_1^2+(\xi-\xi_1)^2-n_1^2-(n-n_1)^2
\xi_1 n_1+(\xi-\xi_1)(n-n_1)
|\leq C,\,\, n_1\neq n-n_1\\
 |n_1|\leq 2N,\,\,\, |n-n_1|\leq 2N,\,\,\,
  |\xi_1|\leq 2N,\,\,\,  |\xi-\xi_1|\leq 2N
\big\},
\end{multline*}and where $\tau, \xi\in \R$, $n\in \Z$ and $N$ is a dyadic integer. We have
$$
\xi_1 n_1+(\xi-\xi_1)(n-n_1)=
(2n_1-n)\xi_1+\xi n-\xi n_1.
$$
The key point is that for $(\xi_1,n_1)\in A_N(\tau, \xi, n)$ we have that $2n_1-n\neq 0$. 
Also, we have that if $A_N(\tau, \xi, n)$ is nonempty then $|n|\leq 4N$. 
For each fixed $n_1\in  A_N(\tau, \xi, n)$ the measure of the possible $\xi_1$ is bounded by
$$
\frac{C}{2n_1-n}
$$
and therefore 
$$
{\rm mes}(A_{N}(\tau,\xi,n))\lesssim \sum_{|n_1|\leq 2N, 2n_1\neq n}\frac{1}{2n_1-n},
$$
where $n$ is an integer in the interval $[-4N,4N]$. Therefore 
$$
{\rm mes}(A_{N}(\tau,\xi,n))\lesssim \sum_{n_1=1}^{10N}\frac{1}{n_1}\lesssim \log(N) 
$$
which implies that 
$$
\Vert u_1\Vert_{L^2_{t,x,y}([0,1]\times \mathbb{R}\times\mathbb{T})}^2\lesssim \log(N) \| f\|_{L^2(\R\times\T)}^4\,.
$$
This completes the proof of Theorem~\ref{meanzerothm}.
\section{Global Strichartz estimates}\label{GSE}
In this section, we intend to deduce the Strichartz estimates \eqref{Stricremove} using the $\epsilon$-removal argument of Barron \cite{B21}. As the argument applies directly to our case, we skip it; see Remark \ref{disscus} for a discussion. The main result of the current section is the following global-in-time Strichartz estimate:
\begin{theorem}\label{Stric}
For all $\epsilon>0$, $N\geq 1$ and $f\in L^2(\mathcal{M})$, there exists $C_{\epsilon}<\infty$ such that 
\begin{equation}
\begin{split}\label{Stric1}
\Vert e^{it\Box} P_{\leq N}f\Vert_{\ell^8L^4(\mathbb{R},L^4(\mathcal{M}))}&\le  C_{\epsilon}N^{\epsilon}\Vert f\Vert_{L^2(\mathcal{M})}.
\end{split}
\end{equation}
\end{theorem}
We also have the following estimate, which follows from $TT^*$ argument as in \cite{GV95, HP14}. We give the proof for the readers' convenience. 
\begin{proposition}\label{Stricinfity}
For all $N\geq 1$, $f\in L^2(\mathcal{M})$, there exists $C<\infty$ such that 
\begin{equation}
\begin{split}\label{Stricinfity1}
\Vert e^{it\Box} P_{\leq N}f\Vert_{\ell^4L^{\infty}(\mathbb{R},L^{\infty}(\mathcal{M}))}&\le  C N\Vert f\Vert_{L^2(\mathcal{M})}. 
\end{split}
\end{equation}
\end{proposition}
Therefore, interpolating \eqref{Stric1} and \eqref{Stricinfity1} yields
\begin{corollary}
Assume $\frac{2}{q}+\frac{1}{p}=\frac{1}{2}$ and $4< q\leq 8$. Then, for all $\epsilon>0$, $N\geq 1$ and $f\in L^2(\mathcal{M})$, there exists $C_{\epsilon}<\infty$ such that 
\begin{equation}
\begin{split}\label{Stricgeneral}
\Vert e^{it\Box} P_{\leq N}f\Vert_{\ell^qL^{p}(\mathbb{R},L^{p}(\mathcal{M}))}&\leq  C_{\epsilon} N^{1-\frac{4}{p}+\epsilon}\Vert f\Vert_{L^2(\mathcal{M})}. 
\end{split}
\end{equation}
\end{corollary}
 When $p>4$, the loss of $N^{\epsilon}$ in \eqref{Stricgeneral} can be removed, which is addressed in Proposition \ref{removeglobalstric}. In particular, Proposition \ref{Stricinfity} affirms the removal of $N^{\epsilon}$ in the case $(p,q)=(\infty,4)$.
\begin{remark}\label{disscus} The proof of Proposition \ref{removeglobalstric} is a particular case of the $\epsilon$-removal argument of Barron \cite{B21}. The main point is to split the bilinear form arising from duality into sums over small and large time scales $(\text{as in}\,\, \eqref{sigmadiagnondiag})$. The small portion is handled via Proposition \ref{epsilonremovalprop} while the sum corresponding to the large time scales is decomposed dyadically $(\text{in the parameter}\,\, \gamma\, \text{with}\,\, |\gamma|\geq 10)$ so that via an interpolation argument $(\text{in a neighborhood of}\,\, (\frac{1}{p},\frac{1}{p}))$ one can obtain a good estimate on the terms of decomposition. This interpolation argument indeed brings about some nice power of $N$ yet still with $\epsilon$-loss and also some power of dyadics due to the decomposition of large time scales. Thus, thanks to these nice powers, the ${\epsilon}$-loss from the power of $N$ can be removed by an easy observation $($see Lemma 4.7 in \cite{B21}$)$, which is valid for $p>4$ close enough to $p^*=4$; but then using Proposition \ref{Stricinfity}, we directly extend the $\epsilon$-removal argument to the remaining range of $p>4$ by an interpolation. Yet another benefit of the nice powers resulting from interpolation argument is that together with the atomic decomposition of Keel and Tao \cite{KT98} on the input functions it allows us to sum in powers of dyadics associated with the decomposition over time scales. For the details, see \cite[Section 4]{B21}. 
\end{remark} 
\begin{proof}[Proof of Proposition \ref{Stricinfity}]
By duality, it suffices to show the following estimate:
\begin{align}\label{L2l43L1}
    \Big\Vert \int_{\mathbb{R}}e^{-is\Box_{\mathbf{x}}} P_{\leq N}f(s,\mathbf{x})\,ds\Big\Vert_{L^2_{\mathbf{x}}(\mathcal{M})}\lesssim N \Vert f\Vert_{\ell^{\frac{4}{3}}_{\gamma}L^1_{t,\mathbf{x}}([\gamma,\gamma+1)\times \mathcal{M})}.
\end{align}
Given a smooth partition of unity $\psi_{\gamma}(t):=\psi(t-\gamma)$ such that $$\sum_{\gamma \in \mathbb{Z}}\psi_{\gamma}(t)=1\,\,\text{and}\,\,\supp \psi \subseteq [-1,1],$$ we denote $f_{\alpha}(t):=\psi(t)f(t+\alpha)$ and write  
\begin{equation}\label{sigmadiagnondiag}
\begin{aligned}
\relax[\text{LHS}\,\eqref{L2l43L1}]^2&=\iint_{\mathbb{R}^2}\langle e^{-is\Box_{\mathbf{x}}} P_{\leq N}f(s,\mathbf{x}), e^{-it\Box_{\mathbf{x}}} P_{\leq N}f(t,\mathbf{x})\rangle_{L^2_{\mathbf{x}}}\,dsdt\\&=\sum_{\alpha, \beta}\iint_{[-1,1]^2}\langle e^{-i(s+\alpha-\beta)\Box_{\mathbf{x}}} P_{\leq N}f_{\alpha}(s,\mathbf{x}), e^{-it\Box_{\mathbf{x}}} P_{\leq N}f_{\beta}(t,\mathbf{x})\rangle_{L^2_{\mathbf{x}}}\,dsdt\\&=\sigma_{\text{diag}}+\sigma_{\text{non-diag}}
\end{aligned}
\end{equation}
where
\begin{equation*}
\begin{aligned}
\sigma_{\text{diag}}&=\sum_{\alpha, |\gamma|\leq 10}\iint_{[-1,1]^2}\langle e^{-i(s-\gamma)\Box_{\mathbf{x}}} P_{\leq N}f_{\alpha}(s,\mathbf{x}), e^{-it\Box_{\mathbf{x}}} P_{\leq N}f_{\alpha+\gamma}(t,\mathbf{x})\rangle_{L^2_{\mathbf{x}}}\,dsdt\\ \sigma_{\text{non-diag}}&=\sum_{\alpha, |\gamma|\geq 10}\iint_{[-1,1]^2}\langle e^{-i(s-t-\gamma)\Box_{\mathbf{x}}} P_{\leq N}f_{\alpha}(s,\mathbf{x}), P_{\leq N}f_{\alpha+\gamma}(t,\mathbf{x})\rangle_{L^2_{\mathbf{x}}}\,dsdt.
\end{aligned}
\end{equation*}
We have
\begin{equation}\label{sigmadiagest}
\begin{aligned}
\sigma_{\text{diag}}\leq \sum_{\alpha, |\gamma|\leq 10}\Big\Vert \int_{\mathbb{R}}e^{-is\Box_{\mathbf{x}}}P_{\leq N}f_{\alpha+\gamma}(s,\mathbf{x})\,ds\Big\Vert_{L^2_{\mathbf{x}}}^2
\end{aligned}.
\end{equation}
Note that for any time interval $I\subset\mathbb{R}$ and $g\in L^2(\mathcal{M})$, Bernstein's inequality yields
\begin{align*}
    \Vert e^{it\Box}P_{\leq N}g\Vert_{L^{\infty}_{t,\mathbf{x}}(I\times\mathcal{M})}\lesssim N \Vert g\Vert_{L^2_{\mathbf{x}}(\mathcal{M})}.
\end{align*}
Therefore, for $h$ supported in $I$, this yields by duality 
\begin{align}\label{L2L1}
    \Big\Vert \int_{\mathbb{R}}e^{-it\Box_{\mathbf{x}}}P_{\leq N}h(t, \mathbf{x})\,dt\Big\Vert_{L^2_{\mathbf{x}}(\mathcal{M})}\lesssim N \Vert h\Vert_{L^1_{t,\mathbf{x}}(I\times\mathcal{M})}.
\end{align}
As a result, since $\supp f_{\alpha+\gamma} \subseteq [-1,1]$, from \eqref{sigmadiagest} and \eqref{L2L1}, we obtain
\begin{equation}\label{sigmadiaest}
\begin{aligned}
    \sigma_{\text{diag}}&\lesssim N^2\sum_{\alpha, |\gamma|\leq 10}\Vert f_{\alpha+\gamma}(t)\Vert_{L^1_{t,\mathbf{x}}([-1,1]\times\mathcal{M})}^2\\&\lesssim N^2\sum_{\alpha}\Vert f_{\alpha}(t)\Vert_{L^1_{t,\mathbf{x}}([-1,1]\times\mathcal{M})}^2\\&\sim N^2\Vert f\Vert_{\ell^2_{\gamma}L^1_{t,\mathbf{x}}([\gamma,\gamma+1]\times\mathcal{M})}^2\leq N^2\Vert f\Vert_{\ell^{\frac{4}{3}}_{\gamma}L^1_{t,\mathbf{x}}([\gamma,\gamma+1]\times\mathcal{M})}^2.
\end{aligned}
\end{equation}
For $\sigma_{\text{non-diag}}$, define $K_{N,\gamma}(t,\mathbf{x}):=\mathbf{1}_{[-1,1]}(t)\cdot K_N(t+\gamma,\mathbf{x})$, where $K_N$ is as in \eqref{kernel} (in view of \eqref{kernel}, we define $K_{NI,\gamma}$ and $K_{NS,\gamma}$ in the same way). Thus we have
\begin{equation*}
\begin{aligned}
\sigma_{\text{non-diag}}&=\sum_{\alpha, |\gamma|\geq 10}\int_t\iint_{\mathbf{x}}\Big(\int_se^{i[(t+\gamma)-s]\Box_{\mathbf{x}}} P_{\leq N}f_{\alpha}(s,\mathbf{x})\,ds\Big)\overline{P_{\leq N}f_{\alpha+\gamma}(t,\mathbf{x})}\,d\mathbf{x}dt\\&=c\sum_{\alpha, |\gamma|\geq 10}\int_t\iint_{\mathbf{x}}[K_{N,\gamma}*_{t,\mathbf{x}}P_{\leq N}f_{\alpha}](t,\mathbf{x})\overline{P_{\leq N}f_{\alpha+\gamma}(t,\mathbf{x})}\,d\mathbf{x}dt.
\end{aligned}
\end{equation*}
Let's estimate the convolution first
\begin{equation*}
\begin{aligned}\label{Linftytx}
&\Vert K_{N,\gamma}*_{t,\mathbf{x}}P_{\leq N}f_{\alpha}\Vert_{L^{\infty}_{t,\mathbf{x}}}\lesssim \Vert K_{NI,\gamma}\otimes K_{NS,\gamma}\Vert_{L^{\infty}_{t,\mathbf{x}}}\Vert P_{\leq N}f_{\alpha}\Vert_{L^{1}_{t,\mathbf{x}}}\\&\lesssim N\Vert |t+\gamma|^{-\frac{1}{2}}\Vert_{L^{\infty}_t} \Vert P_{\leq N}f_{\alpha}\Vert_{L^{1}_{t,\mathbf{x}}}\lesssim |\gamma|^{-\frac{1}{2}}N\Vert P_{\leq N}f_{\alpha}\Vert_{L^{1}_{t,\mathbf{x}}}
\end{aligned}
\end{equation*}
where we have used the fact that $\Vert K_{NI,\gamma}\Vert_{L^{\infty}_x}\lesssim |\gamma|^{-\frac{1}{2}}$ and $\Vert K_{NS,\gamma}\Vert_{L^{\infty}_y}\lesssim N$, since $t\in [-1,1]$ and $|\gamma|\geq 10$. This implies that
\begin{equation}\label{sigmanondiaest}
\begin{aligned}
\sigma_{\text{non-diag}}&\lesssim\sum_{\alpha, |\gamma|\geq 10}\Vert K_{N,\gamma}*_{t,\mathbf{x}}P_{\leq N}f_{\alpha}\Vert_{L^{\infty}_{t,\mathbf{x}}}\Vert f_{\alpha+\gamma}\Vert_{L^{1}_{t,\mathbf{x}}}\\&\lesssim N\sum_{\alpha, |\gamma|\geq 10}|\gamma|^{-\frac{1}{2}}\Vert f_{\alpha}\Vert_{L^{1}_{t,\mathbf{x}}}\Vert f_{\alpha+\gamma}\Vert_{L^{1}_{t,\mathbf{x}}}\\&\lesssim N \Vert f\Vert^2_{\ell_{\gamma}^{\frac{4}{3}}L^1_{t,\mathbf{x}}([\gamma,\gamma+1]\times \mathcal{M})}
\end{aligned}
\end{equation}
in the last estimate above we utilize Hardy-Littlewood-Sobolev inequality (Lemma \ref{HLSineq}). Thus, combining \eqref{sigmadiaest} and \eqref{sigmanondiaest}, the estimate \eqref{L2l43L1} follows.
\end{proof}

\subsection{Proof of Theorem \ref{Stric}}\label{subsectionmainthm}
We follow the argument and adopt the notations in \cite{BCP21}. We accordingly carry out an argument on the frequency side and hence consider $f\in L^2(\mathbb{R}\times\mathbb{Z})$ with $\Vert f\Vert_{L^2}=1$. Note that the Fourier transform of such $f$ corresponds to the function in \eqref{Stric1}. Unlike the notations of the previous sections, we proceed with the following convention for Fourier transform (which is defined in this case on $\mathbb{R}\times\mathbb{Z}$): 
\begin{equation*}
\begin{aligned}
\mathcal{F}_{\xi,k}[f](x,y)&=\widehat{f}(x,y)=\sum_{k\in\mathbb{Z}}\int_{\mathbb{R}}f(\xi,k)e^{ix\xi}e^{iky}d\xi,\\ \widecheck{g}(\xi,k)&=\frac{1}{(2\pi)^2}\int_{\mathbb{R}}\int_{y=0}^{2\pi} g(x,y)e^{-ix\xi}e^{-iky}dydx.
\end{aligned}
\end{equation*}
Therefore, by introducing suitable Gaussian cutoff in time, estimate \eqref{Stric1} amounts to showing that
\begin{align}\label{gammasumNeps}
\mathcal{J}:=\sum_{\gamma\in\mathbb{Z}}J_{\gamma}^{8}\lesssim_{\epsilon} N^{\epsilon},
\end{align}
where
\begin{equation}\label{DefJGamma}
J_\gamma:=\Vert e^{-\frac{1}{4}(t-\gamma)^2}e^{it\Box}\mathcal{F}_{\xi,k}\big[\mathds{1}_{\{|\xi|+|k|\lesssim N\}}f\big]\Vert_{L^4_{x,y,t}(\mathbb{R}\times\mathbb{T}\times\mathbb{R})}.
\end{equation}
To this end, we start with some new notations:
\begin{equation}\label{notationfQ}
\begin{split}
{\vec \xi}&=(\xi_1,\xi_2,\xi_3,\xi_4),\qquad\quad {\vec k}=(k_1,k_2,k_3,k_4),\qquad\quad\vec{\iota}=(1,-1,1,-1),\\
 \vec{\xi}\cdot \vec{\iota}&=\xi_1-\xi_2+\xi_3-\xi_4,\qquad\vec{ k}\cdot\vec{\iota}=k_1-k_2+k_3-k_4,\\
f_j&=\mathds{1}_{\{|\xi_j|+|k_j|\lesssim N\}}f(\xi_j,k_j),\,\, j\in\{1,3\},\qquad f_j=\mathds{1}_{\{|\xi_j|+|k_j|\lesssim N\}}\overline{f}(\xi_j,k_j),\,\, j\in\{2,4\},\\
Q( \xi, k)&=|(\xi_1,k_1)|^2_-+|(\xi_3,k_3)|^2_--|(\xi_2,k_2)|^2_--|(\xi_4,k_4)|^2_-
\end{split}
\end{equation}
 where $|(a,b)|^2_-:=|a|^2-|b|^2$. Making the substitution $t\to t+\gamma$ in \eqref{DefJGamma} and using the notations in \eqref{notationfQ} give that 
\begin{equation*}
\begin{split}
J_\gamma^4&=\int_{x,y,t}\left[\sum_{k_1\dots k_4} \int_{\xi_1\dots\xi_4}\Big(\Pi_{j=1}^4 f_j\Big) e^{- t^2}e^{-i(t+\gamma)Q(\xi, k)} e^{i x\vec{\xi}\cdot\vec{\iota} }e^{iy\vec{ k}\cdot\vec{\iota}} d{\vec\xi}\right] dxdydt\\
&=4\pi^\frac{5}{2}\sum_{k_1\dots k_4} \int_{\xi_1\dots\xi_4}\Big(\Pi_{j=1}^4f_j\Big) e^{-\frac{1}{4} (Q(\xi, k))^2}e^{-i\gamma Q(\xi,k)} \delta(\vec{\xi}\cdot\vec{\iota})\delta(\vec{k}\cdot\vec{\iota}) d{\vec \xi}.
\end{split}
\end{equation*}
Next, writing $f^\prime_l$ analogously as in \eqref{notationfQ} and $Q = Q(\xi, k)$, $Q' = Q(\xi', k')$, we obtain
\begin{equation}\label{DefN}
\begin{split}
\mathcal{J}&=16\pi^5\sum_{\substack{k_1\dots k_4\\ k^\prime_1\dots k^\prime_4}} \int_{\substack{\xi_1\dots\xi_4\\\xi_1^\prime\dots\xi_4^\prime}}\Pi_{j=1}^4 f_j\overline{\Pi_{l=1}^4f_l^\prime}\,  e^{-\frac{1}{4} \left[ Q^2+(Q^\prime)^2\right]} \Big(\sum_\gamma e^{-i\gamma \left[ Q-Q^\prime\right]}\Big)\\
&\qquad\qquad\quad\times \delta(\vec{\xi}\cdot\vec{\iota}) \delta(\vec{\xi'}\cdot\vec{\iota})\delta(\vec{k}\cdot\vec{\iota})\delta(\vec{k'}\cdot\vec{\iota}) d{\vec \xi}d{\vec \xi^\prime}\\&=16\sqrt{2}\pi^{\frac{11}{2}}\sum_{\substack{k_1\dots k_4\\ k^\prime_1\dots k^\prime_4}} \int_{\substack{\xi_1\dots\xi_4\\\xi_1^\prime\dots\xi_4^\prime}}\Pi_{j=1}^4 f_j\overline{\Pi_{l=1}^4f_l^\prime}\,e^{-\frac{1}{4} \left[  Q^2+(Q^\prime)^2\right]}\\
&\qquad\quad\quad\times  \Big(\sum_{\mu\in 2\pi \mathbb{Z}} \delta(\mu-Q+Q^\prime)\Big)\delta(\vec{\xi}\cdot\vec{\iota}) \delta(\vec{\xi'}\cdot\vec{\iota})\delta(\vec{k}\cdot\vec{\iota})\delta(\vec{k'}\cdot\vec{\iota}) d{\vec \xi}d{\vec \xi^\prime}
\end{split}
 \end{equation}
where we have exploited Poisson summation formula in the second equality in \eqref{DefN}. In what follows, we will write $\mathcal{J}$ in the form of an operator given by a suitable kernel so that we can apply Schur test. In this connection, we rely on the following notations:
\begin{equation*}
\begin{split}
\Xi&:=(\xi_1,\xi_3,\xi_2^\prime,\xi_4^\prime),\qquad\Xi^\prime:=(\xi_2,\xi_4,\xi_1^\prime,\xi_3^\prime),\\
K&:=(k_1,k_3,k_2^\prime,k_4^\prime),\qquad K^\prime:=(k_2,k_4,k_1^\prime,k_3^\prime),\\
F(\Xi,K)&:=f_1f_3f^\prime_2f^\prime_4,\qquad F(\Xi^\prime,K^\prime):=f_2f_4f^\prime_1f^\prime_3,\\
\phi_\mu&:=\mu-Q+Q^\prime=\mu-\vert\Xi\vert^2+\vert K\vert^2+\vert\Xi^\prime\vert^2-\vert K^\prime\vert^2.
\end{split}
\end{equation*}
We rewrite \eqref{DefN} via this reformulation as follows:
\begin{equation*}
\begin{split}
\mathcal{J}
&=16\sqrt{2}\pi^{\frac{11}{2}} \sum_{K,K^\prime\in\mathbb{Z}^4} \int_{\Xi,\Xi^\prime}F(\Xi,K)\overline{F(\Xi^\prime,K^\prime)} \, \mathcal{K}(\Xi,K;\Xi^\prime,K^\prime)\, d\Xi d\Xi^\prime\end{split}
\end{equation*}
where
\begin{equation*}
\begin{split}
\mathcal{K}(\Xi,K;\Xi^\prime,K^\prime)&:=  \mathds{1}_{\{|\Xi^\prime|+|K^\prime|\lesssim N\}} e^{-\frac{1}{4} \left[ Q^2+(Q^\prime)^2\right]} \sum_{\mu\in2\pi \mathbb{Z}}\delta(\phi_\mu)\delta(\vec{\xi}\cdot\vec{\iota}) \delta(\vec{\xi'}\cdot\vec{\iota})\delta(\vec{k}\cdot\vec{\iota})\delta(\vec{k'}\cdot\vec{\iota}).\\
\end{split}
\end{equation*}
 Applying Cauchy-Schwarz in $K$ and $\Xi$ first, and then Schur test, the estimate \eqref{gammasumNeps} boils down to demonstrating:
\begin{equation*}
\begin{split}
\sup_{(\Xi,K)\in\mathbb{R}^4\times\mathbb{Z}^4}\mathcal{I}(\Xi,K) \lesssim (\log N)^2,
\end{split}
\end{equation*}
where
\begin{align*}
\mathcal{I}(\Xi,K):=\sum_{K^\prime\in\mathbb{Z}^4}\int_{\Xi^\prime\in\mathbb{R}^4}\mathcal{K}(\Xi,K;\Xi^\prime, K^\prime)\,d\Xi^\prime.	
\end{align*}
In other words, we want to show that $\mathcal{I}(\Xi,K)\lesssim (\log N)^2$ uniformly in $(\Xi,K)\in\mathbb{R}^4\times\mathbb{Z}^4$. To make some simplifications in the expression of $\mathcal{I}(\Xi,K)$, we proceed as follows. Note that for the sum in $\mu$, we make use of the exponential decay coming from the support of $\delta(\phi_{\mu})$, more precisely, we have 
\begin{equation*}
\begin{split}
e^{-\frac{1}{4} \left[ Q^2+(Q^\prime)^2\right]}\le e^{-\frac{1}{16}\mu^2}e^{-\frac{1}{8} \left[ Q^2+(Q^\prime)^2\right]}.
\end{split}
\end{equation*}
Note also that on the supports of $\delta(\vec{\xi}\cdot\vec{\iota})\delta(\vec{k}\cdot\vec{\iota})$ and $\delta(\vec{\xi'}\cdot\vec{\iota})\delta(\vec{k'}\cdot\vec{\iota})$ we can insert \begin{align*}(\xi_4,k_4) &=( \xi_1 - \xi_2 + \xi_3,\, k_1 - k_2 + k_3)\\ (\xi_{3}',k_{3}' )&= (-\xi_{1}' + \xi_{2}' + \xi_{4}', -k_{1}' + k_{2}' + k_{4}') \end{align*} into the expressions for $Q$ and $Q^\prime$ to get
\begin{equation*}
\begin{split}
Q&=-2\left[\vert (\xi_2-c_x,k_2-c_y)\vert_-^2-|(R_1,R_2)|_-^2\right],\\Q^\prime&=2\left[\vert (\xi_{1}' -c'_x,k_{1}' -c'_y)\vert^2_--|(R'_1,R'_2)|_-^2\right],\end{split}
\end{equation*}
 and
\begin{equation*}
\begin{split}
(c_x,c_y)&=(\frac{\xi_1+\xi_3}{2},\frac{k_1+k_3}{2}),\qquad R=(R_1,R_2)=(\frac{\xi_1-\xi_3}{2},\frac{k_1-k_3}{2}),\\
(c_{x}',c'_y)&=(\frac{\xi'_2+\xi'_4}{2},\frac{k'_2 + k'_4}{2}),\qquad R'=(R'_1,R'_2)=(\frac{\xi'_2-\xi'_4}{2},\frac{k'_2-k'_4}{2}).
\end{split}
\end{equation*}
Substituting the above expressions for $Q$ and $Q^\prime$ into $\delta(\phi_\mu)$ gives that
\begin{equation*}
\begin{split}
\delta(\phi_{\mu}) = \frac{1}{2}\delta( |(\xi_2 - c_x, k_2 - c_y)|^2_- +  |(\xi'_1 - c'_x, k'_1 - c'_y )|^2_- - A_{\mu}  ), \qquad A_{\mu} = |R|_-^2+|R'|_-^2-\frac{\mu}{2}.
\end{split}
\end{equation*}
The variables $R$ and $R^\prime$ depend only on $(\Xi,K)$, so we may fix them. Therefore with the above notation, we estimate
\begin{equation}\label{DefI}
\begin{aligned}
&\mathcal{I}(\Xi,K)\\&\leq \frac{1}{2}\sup_{\mu\in2\pi\mathbb{Z}}\Big[\sum_{k_2, k'_1} \int_{\mathbb{R}^2}   \mathds{1}_{\{|\xi_2|+|k_2|\lesssim N\}} \mathds{1}_{\{|\xi_1^\prime|+|k_1^\prime|\lesssim N \}} e^{-\frac{1}{2} \left[\vert (\xi_2-c_x,k_2-c_y)\vert_-^2-|R|_-^2\right]^2 }\\ & \qquad \times e^{-\frac{1}{2}\left[\vert (\xi_{1}' -c'_x,k_{1}' -c'_y)\vert^2_--|R'|_-^2\right]^2}\delta( |(\xi_2 - c_x, k_2 - c_y)|^2_- +  |(\xi'_1 - c'_x, k'_1 - c'_y )|^2_- - A_{\mu})  d\xi_2d\xi'_1\Big]\\&\qquad\times\sum_{\mu \in 2\pi \mathbb{Z} } e^{-\frac{1}{16} \mu^2}.
\end{aligned}
\end{equation}
Therefore, we now want to show
\begin{equation}\label{blackI}
\begin{split}
\sum_{|k|,|k^\prime|\lesssim N}\int_{|\xi|,|\xi^\prime|\lesssim N}&e^{-\frac{1}{2}\left[ \vert\vert(\xi,k)-{\vec C}\vert^2_--|R|_-^2\vert^2+\vert \vert(\xi^\prime,k^\prime)-{\vec C}^\prime\vert^2_--|R^\prime|_-^2\vert^2\right]}\\&\times\delta(\vert (\xi,k)-{\vec C}\vert^2_-+\vert (\xi^\prime,k^\prime)-{\vec C^\prime}\vert^2_--A)d\xi d\xi^\prime\lesssim (\log N)^2
\end{split}
\end{equation}
uniformly in $A,R,R^\prime\in\mathbb{R}$, $\vec C, \vec C^\prime\in\mathbb{R}^2$. Notice that, in view of \eqref{DefI}, the second components of $\vec{C}, \vec{C'}$ are in $\frac{1}{2}\mathbb{Z}$. By translation invariance, we may suppose that ${\vec C}=(0,c)$, ${\vec C}^\prime=(0,c^\prime)$ for $c,c^\prime\in\{0, \frac{1}{2} \}$. To proceed further, we define the sets 
\begin{equation}\label{sets}
    \begin{aligned}
        A_j:=\{(\xi,k)\in\mathbb{R}\times\mathbb{Z}: |\Phi(\xi,k)|\in [j,j+1),\,|\xi|+|k|\lesssim N\},\quad j\geq 0,
    \end{aligned}
\end{equation}
where $\Phi(\xi,k):=|(\xi,k)- \vec C|^2_--|R|^2_-$. Analogously, we define the sets $A_j'$ with $\Phi'(\xi',k')$. Therefore, by symmetry, \eqref{blackI} is equivalent to demanding that
\begin{equation}\label{sumI_jp}
    \begin{aligned}
   \sum_{j,p}{\bf I}_{jp} \lesssim (\log N)^2   
    \end{aligned}
\end{equation}
where
\begin{equation*}
    \begin{aligned}
   {\bf I}_{jp}:=\sum_{k,k^\prime}\int_{\xi,\xi^\prime}&\mathds{1}_{A_j}(\xi,k)\mathds{1}_{A_p'}(\xi^\prime,k^\prime)\mathds{1}_{\{|\xi^\prime|\leq |\xi|\}}e^{-\frac{1}{2}\left[ |\Phi(\xi,k)|^2+|\Phi'(\xi',k')|^2\right]}\\&\times\delta(\vert (\xi,k)-{\vec C}\vert^2_-+\vert (\xi^\prime,k^\prime)-{\vec C^\prime}\vert^2_--A)d\xi d\xi^\prime . 
    \end{aligned}
\end{equation*}
Fix $(j,p)$. Note that there are at most finitely many $k$ and $k'$ in the sums corresponding to the integration region $|\xi|, |\xi'|\leq 2$, which is due to the definition of the sets $A_j$ and $A_p^\prime$ and the fact that $R, R'$ have been fixed. Therefore, we may write 
\begin{equation}\label{I_jp}
\begin{aligned}
   &{\bf I}_{jp}\lesssim \\&e^{-\frac{1}{2}(j^2+p^2)}\Big[\sum_{k,k^\prime=O(1)}\int_{|\xi|, |\xi'|\leq 2}\mathds{1}_{A_j}(\xi,k)\mathds{1}_{A_p'}(\xi^\prime,k^\prime)\delta(|\xi|^2+|\xi^\prime|^2-|k-c|^2-|k'-c'|^2-A)d\xi d\xi^\prime \\&+\sum_{k'}\int_{\mathbb{R}}\mathds{1}_{A_p'}(\xi^\prime,k^\prime)\Big(\sum_{|k|\lesssim N}\int_{|\xi|\geq 2}\mathds{1}_{A_j}(\xi,k)\delta(|\xi|^2+|\xi^\prime|^2-|k-c|^2-|k'-c'|^2-A)d\xi \Big)d\xi'\Big]\\&=:e^{-\frac{1}{2}(j^2+p^2)}[ {\bf I}_{jp}^1+ {\bf I}_{jp}^2].
    \end{aligned}
\end{equation}
 To control the above terms, we need the following lemma : 
 \begin{lemma}\label{simpleLemma}
    Let $C_0\in \mathbb{R}$, $N,K\geq 1$ be constants. Then,
    \begin{equation}\label{measure}
    \begin{aligned}
        \sup_{\substack{|a|, |b|\lesssim N}}{\rm mes}(\{(\xi,k)\in \mathbb{R}\times \mathbb{Z}: C_0\leq (\xi-a)^2- (k-b)^2\leq C_0+K,\, |\xi|+|k|\lesssim N\})\lesssim K\log N
    \end{aligned}
    \end{equation}
    where ${\rm mes}$ denotes the product measure of Lebesgue and the counting measures. The implicit constant is independent of $C_0$. Therefore, the sets $A_j$ in \eqref{sets} satisfy ${\rm mes}(A_j)\lesssim \log N$. Moreover, we have
    \begin{equation}\label{SimpleBd1}
\begin{split}
\sup_{A\in\mathbb{R}}\iint_{\mathbb{R}^2} \delta(\zeta^2+\eta^2-A)d\zeta d\eta& =\pi.
\end{split}
\end{equation}
\end{lemma}
\begin{proof}
Note that \eqref{measure} is nothing but restating the estimate for \eqref{3} (by translation the statements are equivalent), see the proof of Theorem \ref{mainthmL4est}. Here we only prove \eqref{SimpleBd1}, which is a specific consequence of the co-area formula : for any Schwartz function $F\in \mathcal{S}(\mathbb{R}^2)$,
$$ \int_{\mathbb{R}^2}F(\zeta^2+\eta^2) d\zeta d\eta =\int_0^{\infty} F(r)dr \int_0^{2\pi}\frac{r d\theta}{2r}=\pi\int_0^{\infty}F(r)dr.
$$
Hence \eqref{SimpleBd1} follows by taking $F$ as a sequence of approximate identity.
\end{proof}
 The identity \eqref{SimpleBd1} implies that 
\begin{align}\label{I_jp^1}
 {\bf I}_{jp}^1\lesssim \sup_{B\in\mathbb{R}}\iint_{\mathbb{R}^2} \delta(\xi^2+(\xi^\prime)^2-B)d\xi d\xi^\prime \sum_{k,k^\prime=O(1)}1\lesssim 1. 
\end{align}
For the term ${\bf I}_{jp}^2$, we proceed as follows
\begin{equation}\label{I_jp^2}
    \begin{aligned}
        {\bf I}_{jp}^2&\lesssim \Big(\sup_{\substack{B\in\mathbb{R}\\
        C\in\mathbb{R}, c\in\frac{1}{2}\mathbb{Z} }}\sum_{|k|\lesssim N}\int_{\mathbb{R}}\mathds{1}_{\big[\sqrt{|k-c|^2+C},\sqrt{|k-c|^2+C+1}\big)}(\xi)\delta(\xi^2-B) d\xi\Big)\sum_{k'}\int_{\mathbb{R}}\mathds{1}_{A_p'}(\xi^\prime,k^\prime)d\xi'\\&={\rm mes}(A_p')\Big(\sup_{\substack{B\in\mathbb{R}\\
C\in\mathbb{R},c\in\frac{1}{2}\mathbb{Z}    
}}\sum_{|k|\lesssim N}\mathds{1}_{\big[\sqrt{|k-c|^2+C},\sqrt{|k-c|^2+C+1}\big)}(\sqrt{B})\frac{1}{2\sqrt{B}}\Big)\\&\lesssim \log N\sup_{C\in\mathbb{R},c\in\frac{1}{2}\mathbb{Z}}\sum_{|k|\lesssim N}\frac{
\mathbf{1}_{ \sqrt{|k-c|^2+C }>1  }
 }{\sqrt{|k-c|^2+C}}\lesssim  (\log N)^2.
    \end{aligned}
\end{equation}
Note that in the estimate above we have used \eqref{measure} and the following identity
\begin{equation*}
    \begin{aligned}
        \int_{-\infty}^{\infty}f(x)\delta(g(x))dx=\sum_i\frac{f(x_i)}{|g'(x_i)|}
    \end{aligned}
\end{equation*}
where $g(x_i)=0$ with $g'(x_i)\neq 0$. Note also that $\sqrt{|k-c|^2+C}>1$, since in the support of $\xi$-integration region of ${\bf I}_{jp}^2$, we have the relation that $|k-c|^2+C+1\geq |\xi|^2$ and $|\xi|\geq 2$. Consequently, combining \eqref{I_jp}, \eqref{I_jp^1}, and \eqref{I_jp^2}, the estimate \eqref{sumI_jp} follows. 
\section{Proof of Theorem \ref{maintm}}\label{SectNonlinear}
In this section, based on the theory of critical function spaces, we establish the well-posedness theory for \eqref{eq:HNLS}, and for this purpose, we start by proving an associated nonlinear estimate. The following lemma will be of fundamental use in such an estimate, where, using the atomic structure of $U^p$, we transfer the Strichartz estimate \eqref{Stricremove} in Proposition \ref{removeglobalstric}  to an estimate with respect to the function spaces introduced in Section \ref{preliminaries}.  
\begin{lemma}\label{Strichartzkilvis} 
Let $I\subset \mathbb{R}$ be a time interval. For all $C\in \mathscr{C}_N$ and admissible pair $(p,q)$, we have
\begin{equation}\label{StricartzwithCcubes}
\begin{aligned}
  \Vert \mathds{1}_{I}\cdot P_C u\Vert_{\ell_{\gamma}^qL^p([\gamma,\gamma+1]\times \mathbb{R}\times \mathbb{T})} \lesssim N^{1-\frac{4}{p}}\Vert P_C u\Vert_{U^{\min\{p,q\}}_{\Box}(I;L^2)}\lesssim N^{1-\frac{4}{p}}\Vert P_C u\Vert_{Y^0(I)},
\end{aligned}
\end{equation}
where the implicit constants are independent of the time interval $I$.
\end{lemma}
\begin{proof}
The second inequality follows from the property of $X^0,Y^0$ (see Remark \ref{remark2.6}), so we only prove the first one. Denote $r=\min\{p,q\}$. By definition, it suffices to prove the estimate when $e^{-it\Box}P_Cu$ is an atom on $I$ with a partition $\bigcup_{k=1}^K[t_{k-1},t_k)$, i.e. 
$$ P_Cu=\sum_{k=1}^K \mathds{1}_{[t_{k-1},t_k)}e^{it\Box}P_C\phi_k,\quad \sum_{k=1}^K\|P_C\phi_k\|_{L_{\mathbf{x}}^2}^r=1.
$$
In the following, we use the notation $\ell_k^p$ for $1\leq k\leq K$. Then, it follows that
\begin{align*}
\|\mathds{1}_I(t)P_Cu\|_{L^p_{t,\mathbf{x}}([\gamma,\gamma+1]\times\mathcal{M})}\sim&\|\mathds{1}_{[t_{k-1},t_k)}e^{it\Box}P_C\phi_k \|_{\ell_{k}^pL^p_{t,\mathbf{x}}([\gamma,\gamma+1]\times\mathcal{M})}\\ \leq
&\|\mathds{1}_{[t_{k-1},t_k)}e^{it\Box}P_C\phi_k \|_{\ell_{k}^rL^p_{t,\mathbf{x}}([\gamma,\gamma+1]\times\mathcal{M})}. 
\end{align*}
Taking the $\ell^q$-norm for $\gamma\in\mathbb{Z}$
on both sides, we obtain by Minkowski and Proposition 
\ref{removeglobalstric} that 
\begin{align*} \|\mathbf{1}_I(t)P_Cu\|_{\ell_{\gamma}^qL^p_{t,\mathbf{x}}([\gamma,\gamma+1]\times\mathcal{M})}\leq &\|\mathds{1}_{[t_{k-1},t_k)}e^{it\Box}P_C\phi_k \|_{\ell_k^r\ell_{\gamma}^qL^p_{t,\mathbf{x}}([\gamma,\gamma+1]\times\mathcal{M})}\\
\lesssim & N^{1-\frac{4}{p}}
\|P_C\phi_k\|_{\ell_k^rL_{\mathbf{x}}^2}=N^{1-\frac{4}{p}}.
\end{align*}
This completes the proof of Lemma \ref{StricartzwithCcubes}.
\end{proof}
In the following lemma, we take into account $k\geq2$, as the global Strichartz estimates \eqref{Stricremove} are in full effect for such powers; for the cubic case, we can exploit Theorem \ref{mainthmL4est} to achieve the result of Theorem \ref{maintm}, but we skip the details.
\begin{lemma}\label{lemma1}
Let $I\subset \mathbb{R}$. Assume that $k\geq 2$ is fixed and $s\geq s_{2,k}=1-\frac{1}{k}$. Then we have
\begin{align}\label{mullinest}
\Big|\int_I\int_{\mathcal{M}}\prod_{l=1}^{2k+1}\widetilde{u}_l\overline{v}\,dxdt\Big|\lesssim \Vert v\Vert_{Y^{-s}(I)}\Vert u_1\Vert_{Y^{s}(I)}\prod_{l=2}^{2k+1}\Vert u_l\Vert_{Y^{s_{2,k}}(I)}
\end{align}
where $\widetilde{u}_l\in \{u_l, \overline{u}_l\}$, and the implicit constants are independent of the time interval $I$.
\end{lemma}
\begin{proof}
 When $k=2$ the proof follows as that in \cite[Lemma 5.2]{B21}, so we consider $k\geq3$. In what follows, we frequently ignore the time domain (as well as the spatial domain) in notations, but it must be understood that all the norms we consider here are taken over the time domain $I$ (where $I=\mathbb{R}$ is also allowed). We shall next introduce the mixed-norm H\"older's inequality exponents to be used in the proof of \eqref{mullinest}, which are given by some restrictions. To this end, let $p$ satisfy the following  
\begin{equation}\label{psatisfy}
\begin{aligned}
4<p<\frac{8k}{k+1}\quad \text{for}\,\,k\geq 3,
\end{aligned}
\end{equation}
where the upper bound we choose ensures that
\begin{align*}\label{s_nm<0}
  2-\frac{8}{p}-s_{2,k}<0.  
\end{align*}
Denote $\delta:=-2+\frac{8}{p}+s_{2,k}>0$ and the exponent associated with $s_{2,k}$ by $p_{2,k}:=4k$. Also let $q_{2,k}:=\frac{4p_{2,k}}{p_{2,k}-2}$ so that the pair $(p_{2,k},q_{2,k})$ is admissible. Then, by Lemma \ref{Strichartzkilvis}, we have
\begin{equation}\label{pnmsnm}
\begin{aligned}
  \Vert \mathds{1}_I\cdot P_Nu\Vert_{\ell^{q_{2,k}}L^{p_{2,k}}([\gamma,\gamma+1]\times \mathcal{M})} \lesssim N^{s_{2,k}}\Vert P_Nu\Vert_{Y^0(I)}. 
\end{aligned}
\end{equation}
Hence, to obtain \eqref{mullinest}, we apply H\"older's inequality with the following $L^p$ exponents
\begin{align}\label{pp_nmq2}
    \frac{2}{p}+\frac{k-2}{p_{2,k}}+\frac{1}{\widetilde{p}}=\frac{1}{2},
\end{align}
where $p$ is as in \eqref{psatisfy} and $p_{2,k}$ is as introduced above. Note that $\widetilde{p}>p_{2,k}\geq4$ (for $k\geq 1$) and
\begin{align}\label{necessarydelta}
  1-\frac{4}{\widetilde{p}}-s_{2,k}=-2+\frac{8}{p}+s_{2,k}=\delta>0.
\end{align}
Next, we assume that $(p,q)$ and $(\widetilde{p},\widetilde{q})$ are admissible pairs as in Definition \ref{admissibledefn}. Since we have that
\begin{align*}
    \frac{2}{q}+\frac{k-2}{q_{2,k}}+\frac{1}{\widetilde{q}}\neq\frac{1}{2},
\end{align*}
we instead take $\ell^q$ H\"older exponents as $(8, 8(k-2),8)$ to reach the desired equality in \eqref{pp_nmq2} when substituting these exponents for $(p,p_{2,k},\widetilde{p})$ in \eqref{pp_nmq2}. Then, exploiting the embeddings $\ell^q, \ell^{\widetilde{q}}\hookrightarrow \ell^8$ $(\text{since}\,\,p,\widetilde{p}>4)$, and $\ell^{q_{2,k}}\hookrightarrow \ell^{8(k-2)}$ (for $k\geq3$), we still benefit from the global Strichartz estimates \eqref{StricartzwithCcubes} with the aforementioned admissible pairs. By Littlewood-Paley decomposition, we write 
$$ u_l=\sum_{N_l}P_{N_l}u_l,\quad l\geq 1,\quad  v=\sum_{N_0}P_{N_0}v,
$$
where the summation is taken over dyadic numbers $N_0,N_l\geq 1$. Expanding the left hand side of  \eqref{mullinest}, by symmetry, it suffices to consider the following cases.
\vspace{0.12cm}
\noindent  
\\{\bf Case A.} $N_0\sim N_1\geq N_2\geq \cdots\geq N_{2k+1}$.\vspace{0.12cm} \\
Let $\{C_{\alpha}\}$ be a cube partition of $\mathbb{R}\times \mathbb{Z}$ into cubes of size $N_2$ and $\{C_{\beta}\}$ be a cube partition of size $N_3$. Since there are finitely many cubes of side length $\sim N_2$ that cover the spatial Fourier support of $P_{N_2}u_2\prod_{l=2}^{k}P_{N_{2l}}u_{2l}$, we see that the terms 
$P_{C_{\alpha}}P_{N_0}v\prod_{l=1}^{k}P_{N_{2l}}u_{2l}$ are almost orthogonal in $L^2_{x,y}$. The same applies to $P_{C_{\beta}}P_{N_1}v\prod_{l=1}^{k}P_{N_{2l+1}}u_{2l+1}$. So we have:
\begin{equation}\label{beforedecomp}
\begin{aligned}
&\Big|\int_{I}\int_{\mathcal{M}}\prod_{l=1}^{2k+1}\widetilde{u}_l\overline{v}\,dxdydt\Big|\lesssim \sum_{N_0\sim N_1\geq \cdots \geq N_{2k+1}} \Big\Vert P_{N_0}v\prod_{l=1}^{2k+1}P_{N_l}u_{l}\Big\Vert_{L^1_{t,x,y}}\\&\lesssim  \sum_{N_0\sim N_1\geq \cdots \geq N_{2k+1}} \Bigg(\sum_{\alpha\in\mathbb{Z}}\Big\Vert P_{C_{\alpha}}P_{N_0}v\prod_{l=1}^{k}P_{N_{2l}}u_{2l} \Big\Vert_{L^2_{t,x,y}}^2\Bigg)^{\frac{1}{2}}\\&\quad\quad\times\Bigg(\sum_{\beta\in\mathbb{Z}}\Big\Vert P_{C_{\beta}}P_{N_{1}}u_{1}\prod_{l=1}^{k}P_{N_{2l+1}}u_{2l+1}  \Big\Vert_{L^2_{t,x,y}}^2\Bigg)^{\frac{1}{2}}.
\end{aligned}
\end{equation}
For $\gamma\in\mathbb{Z}$, set $I_{\gamma}:=I\cap{[\gamma,\gamma+1)}$ and decompose 
\begin{align*}
    I=\bigcup_{\gamma\in\mathbb{Z}}I_{\gamma}
\end{align*}
 to estimate $L^2$ norm in \eqref{beforedecomp} (by using the fact that the intervals $I_{\gamma}$ are disjoint):
 \begin{equation}
\begin{aligned}\label{L^2first}
& \Big\Vert P_{C_{\alpha}}P_{N_0}v\prod_{l=1}^{k}P_{N_{2l}}u_{2l} \Big\Vert_{L^2_{t,x,y}}^2=\sum_{\gamma\in \mathbb{Z}}  \Big\Vert P_{C_{\alpha}}P_{N_0}v\prod_{l=1}^{k}P_{N_{2l}}u_{2l} \Big\Vert_{L^2_{t,\mathbf{x}}(I_{\gamma}\times \mathcal{M})}^2\\&\lesssim \Big[\Vert P_{C_{\alpha}}P_{N_0}v\Vert_{\ell^8_{\gamma}L^p_{t,\mathbf{x}}(I_{\gamma}\times \mathcal{M})}\Vert P_{N_2}u_2\Vert_{\ell^8_{\gamma}L^p_{t,\mathbf{x}}(I_{\gamma}\times \mathcal{M})}\\&\qquad\qquad\times\prod_{l=2}^{k-1}\Vert P_{N_{2l}}u_{2l}\Vert_{\ell^{8(k-2)}_{\gamma}L^{p_{2,k}}_{t,\mathbf{x}}(I_{\gamma}\times \mathcal{M})}\Vert P_{N_{2k}}u_{2k}\Vert_{\ell^8_{\gamma}L^{\widetilde{p}}_{t,\mathbf{x}}(I_{\gamma}\times \mathcal{M})}\Big]^2\\&\lesssim \Big[\Vert P_{C_{\alpha}}P_{N_0}v\Vert_{\ell^q_{\gamma}L^p_{t,\mathbf{x}}(I_{\gamma}\times \mathcal{M})}\Vert P_{N_2}u_2\Vert_{\ell^q_{\gamma}L^p_{t,\mathbf{x}}(I_{\gamma}\times \mathcal{M})}\\&\qquad\qquad\times\prod_{l=2}^{k-1}\Vert P_{N_{2l}}u_{2l}\Vert_{\ell^{q_{2,k}}_{\gamma}L^{p_{2,k}}_{t,\mathbf{x}}(I_{\gamma}\times \mathcal{M})}\Vert P_{N_{2k}}u_{2k}\Vert_{\ell^{\widetilde{q}}_{\gamma}L^{\widetilde{p}}_{t,\mathbf{x}}(I_{\gamma}\times \mathcal{M})}\Big]^2\\&\lesssim \Big[ N_2^{2-\frac{8}{p}-s_{2,k}}N_{2k}^{1-\frac{4}{\widetilde{p}}-s_{2,k}}\Vert P_{C_{\alpha}}P_{N_0}v\Vert_{Y^0}\prod_{l=1}^{k}\Vert P_{N_{2l}}u_{2l}\Vert_{Y^{s_{2,k}}}\Big]^2
\end{aligned}
\end{equation}
where we have used H\"older's inequality with exponents as in \eqref{pp_nmq2}, the embeddings $\ell^q, \ell^{\widetilde{q}}\hookrightarrow \ell^8$, $\ell^{q_{2,k}}\hookrightarrow \ell^{8(k-2)}$, and the estimates \eqref{StricartzwithCcubes}, \eqref{pnmsnm}. Similarly, we have
\begin{equation}
\begin{aligned}\label{L^2second}
& \Big\Vert P_{C_{\beta}}P_{N_1}u_1\prod_{l=1}^{k}P_{N_{2l+1}}u_{2l+1} \Big\Vert_{L^2_{t,x,y}}\lesssim N_3^{2-\frac{8}{p}-s_{2,k}}N_{2k+1}^{1-\frac{4}{\widetilde{p}}-s_{2,k}}\Vert P_{C_{\beta}}P_{N_1}u_1\Vert_{Y^0}\prod_{l=1}^{k}\Vert P_{N_{2l+1}}u_{2l+1}\Vert_{Y^{s_{2,k}}}.
\end{aligned}
\end{equation}
Therefore, inserting \eqref{L^2first} and \eqref{L^2second} into \eqref{beforedecomp} and using the expression \eqref{necessarydelta}, we get
\begin{align*}
&\text{RHS}\,\eqref{beforedecomp}\lesssim\\&\sum_{N_0\sim N_1\geq \cdots \geq N_{2k+1}}  \Big(\frac{N_{2k}}{N_2}\Big)^{\delta}\Big(\frac{N_{2k+1}}{N_3}\Big)^{\delta}\Big(\sum_{\alpha\in\mathbb{Z}}\Vert P_{C_{\alpha}}P_{N_0}v\Vert_{Y^0}^2\Big)^{\frac{1}{2}} \Big(\sum_{\beta\in\mathbb{Z}}\Vert P_{C_{\beta}}P_{N_1}u_1\Vert_{Y^0}^2\Big)^{\frac{1}{2}}\prod_{l=2}^{2k+1}\Vert P_{N_{l}}u_{l}\Vert_{Y^{s_{2,k}}} \\&\sim \sum_{N_0\sim N_1\geq \cdots \geq N_{2k+1}} \Big(\frac{N_{2k}}{N_2}\Big)^{\delta}\Big(\frac{N_{2k+1}}{N_3}\Big)^{\delta}\Vert P_{N_0}v\Vert_{Y^0}\Vert P_{N_1}u_1\Vert_{Y^0}\prod_{l=2}^{2k+1}\Vert P_{N_{l}}u_{l}\Vert_{Y^{s_{2,k}}}\\&\lesssim \sum_{\substack{N_0, N_1\\ N_0\sim N_1}}\Big(\frac{N_{0}}{N_1}\Big)^{s}\Vert P_{N_0}v\Vert_{Y^{-s}}\Vert P_{N_1}u_1\Vert_{Y^s}\Bigg(\sum_{\substack{N_2, N_3\\ N_2\geq N_3}}\Big(\frac{N_{3}} {N_2}\Big)^{\frac{\delta}{2k}}\Vert P_{N_2}u_2\Vert_{Y^{s_{2,k}}}\Vert P_{N_3}u_3\Vert_{Y^{s_{2,k}}}\Bigg)\cdots \\&\quad \quad\times\Bigg(\sum_{\substack{N_{2k}, N_{2k+1}\\ N_{2k}\geq N_{2k+1}}}\Big(\frac{N_{2k+1}}{N_{2k}}\Big)^{\frac{\delta}{2k}}\Vert P_{N_{2k}}u_{2k}\Vert_{Y^{s_{2,k}}}\Vert P_{N_{2k+1}}u_{2k+1}\Vert_{Y^{s_{2,k}}}\Bigg) \\&\lesssim \Big(\sum_{N_0}\Vert P_{N_0}v\Vert_{Y^{-s}}^2\Big)^{\frac{1}{2}}\Big(\sum_{N_1}\Vert P_{N_1}u_1\Vert_{Y^{s}}^2\Big)^{\frac{1}{2}}\Big(\sum_{\substack{N_2, N_3\\ N_2\geq N_3}}\Big(\frac{N_{3}} {N_2}\Big)^{\frac{\delta}{2k}}\Vert P_{N_2}u_2\Vert_{Y^{s_{2,k}}}^2\Big)^{\frac{1}{2}}\\&\quad\quad\times\Big(\sum_{\substack{N_2, N_3\\ N_2\geq N_3}}\Big(\frac{N_{3}} {N_2}\Big)^{\frac{\delta}{2k}}\Vert P_{N_3}u_3\Vert_{Y^{s_{2,k}}}^2\Big)^{\frac{1}{2}} \cdots  \Big(\sum_{\substack{N_{2k}, N_{2k+1}\\ N_{2k}\geq N_{2k+1}}}\Big(\frac{N_{2k+1}} {N_{2k}}\Big)^{\frac{\delta}{2k}}\Vert P_{N_{2k}}u_{2k}\Vert_{Y^{s_{2,k}}}^2\Big)^{\frac{1}{2}}\\&\quad\quad\times\Big(\sum_{\substack{N_{2k}, N_{2k+1}\\ N_{2k}\geq N_{2k+1}}}\Big(\frac{N_{2k+1}} {N_{2k}}\Big)^{\frac{\delta}{2k}}\Vert P_{N_{2k+1}}u_{2k+1}\Vert_{Y^{s_{2,k}}}^2\Big)^{\frac{1}{2}}\\&\lesssim \Vert v\Vert_{Y^{-s}}\Vert u_1\Vert_{Y^{s}}\prod_{l=2}^{2k+1}\Vert u_l\Vert_{Y^{s_{2,k}}}.
\end{align*}
\noindent 
\\{\bf Case B.} $N_1\sim N_2\geq N_0\geq N_3\geq  \cdots\geq N_{2k+1}$. \vspace{0.12cm}\\
We use the cube decomposition with respect to $N_0$ and $N_3$ size cubes in this case, and then proceed by applying mixed-norm H\"older's inequality with the same exponents as in the previous case, and Strichartz inequalities \eqref{StricartzwithCcubes}, \eqref{pnmsnm}, together with the identity \eqref{necessarydelta} to arrive at
\begin{align*}
&\Big|\int_{I}\int_{\mathcal{M}}\prod_{l=1}^{2k+1}\widetilde{u}_l\overline{v}\,dxdydt\Big|\\&\lesssim \sum_{N_1\sim N_2\geq N_0\geq \cdots \geq N_{2k+1}} \Big\Vert \prod_{l=0}^{k}P_{N_{2l+1}}u_{2l+1} \Big\Vert_{L^2_{t,x,y}}\Big\Vert P_{N_0}v\prod_{l=1}^{k}P_{N_{2l}}u_{2l} \Big\Vert_{L^2_{t,x,y}}\\&\lesssim \sum_{N_1\sim N_2\geq N_0\geq \cdots \geq N_{2k+1}} \Big(\frac{N_{2k}}{N_0}\Big)^{\delta}\Big(\frac{N_{2k+1}}{N_3}\Big)^{\delta}N_0^{s+s_{2,k}}N_1^{-s-s_{2,k}}\Vert P_{N_{0}}v \Vert_{Y^{-s}}\Vert P_{N_1}u_1\Vert_{Y^{s}}\\&\quad\quad\times\Vert P_{N_2}u_2\Vert_{Y^{s_{2,k}}}\prod_{l=3}^{2k+1}\Vert P_{N_{l}}u_{l} \Vert_{Y^{s_{2,k}}}\\&\lesssim \sum_{\substack{N_1,N_2\\N_1\sim N_2}}\Vert P_{N_1}u_1\Vert_{Y^s}\Vert P_{N_2}u_2\Vert_{Y^{s_{2,k}}}\sum_{N_0\geq N_3\geq  \cdots \geq N_{2k+1}}\Big(\frac{N_{2k}}{N_0}\Big)^{\delta}\Big(\frac{N_{2k+1}}{N_3}\Big)^{\delta}\\&\quad\quad\times\Vert P_{N_{0}}v \Vert_{Y^{-s}}\prod_{l=3}^{2k+1}\Vert P_{N_{l}}u_{l} \Vert_{Y^{s_{2,k}}}\\&\lesssim \Vert v\Vert_{Y^{-s}}\Vert u_1\Vert_{Y^{s}}\prod_{l=2}^{2k+1}\Vert u_l\Vert_{Y^{s_{2,k}}},
\end{align*}
where to the last step, the analysis is similar to the previous case. This completes the proof of Lemma \ref{lemma1}
\end{proof}
\begin{proposition}\label{keyproposition}
 Let $s\geq s_{2,k}$ be fixed. Then, for any time interval $I\subset \mathbb{R}$ and $u_j\in X^s(I)$, $j=1,...,2k+1$, we have
 \begin{align}\label{mainest}
 \Big\Vert \int_0^te^{i(t-\tau)\Box} \prod_{j=1}^{2k+1}\widetilde{u}_j\,d\tau\Big\Vert_{X^s(I)}\lesssim \sum_{l=1}^{2k+1}\Vert u_l\Vert_{X^s(
 I)}\prod_{\substack{j=1\\j\neq l}}^{2k+1}\Vert u_j\Vert_{X^{s_{2,k}}(I)}
 \end{align}
 where $\widetilde{u}_j\in \{u_j, \overline{u}_j\}$, and the implicit constants are independent of the time interval $I$.
\end{proposition}
\begin{proof}
    The Estimate \eqref{mainest} follows directly from 
    \begin{equation*} 
 \begin{aligned}
  \Big\Vert \int_0^te^{i(t-\tau)\Box} f(\tau)d\tau\Big\Vert_{X^s(I)} \lesssim \sup_{\Vert v\Vert_{Y^{-s}(I)}=1}\Big|\int_{I}\int_{\mathcal{M}}f(t,\mathbf{x})\overline{v(t,\mathbf{x})}d\mathbf{x}dt\Big|
 \end{aligned}
 \end{equation*}
    (for a proof we refer to \cite{HTT11}) and Lemma \ref{lemma1}.
\end{proof}
\begin{proof}[Proof of Theorem \ref{maintm}]
We follow the argument in \cite{HHK09, HTT11,KV14}. Assume that $k\geq 2$ is fixed and $I\subset \mathbb{R}_{\geq 0}$ is a fixed time interval containing $0$. Let $\delta>0$ and $\epsilon>0$ to be determined later. For convenience let us write $F(u)=|u|^{2k}u$ and denote
\begin{align*}
\mathcal{I}[F(u)](t):=\int_0^te^{i(t-\tau)\Box}F(u)(\tau)\,d\tau.
\end{align*}
Then, given $u_0\in H^{s_{2,k}}(\mathbb{R}\times \mathbb{T})$ with $\Vert u_0\Vert_{H^{s_{2,k}}(\mathbb{R}\times \mathbb{T})}\leq \epsilon$, the contraction argument is applied to the operator
\begin{align}\label{Gammaoperator}
\Gamma(u)(t):=e^{it\Box}u_0\pm i\mathcal{I}[F(u)](t),\quad t\in I\setminus\{0\},
\end{align}
on the ball
\begin{align*}
    B_{\delta}:=\{u\in X^{s_{2,k}}(I)\cap C_t(I;H^{s_{2,k}}_{x,y}(\mathbb{R}\times \mathbb{T})): \Vert u\Vert_{X^{s_{2,k}}(I)}\leq \delta\}.
\end{align*}
 Thus, for $u\in B_{\delta}$, by Lemma \ref{linearestlemma} and Proposition \ref{keyproposition}, we obtain
\begin{align*}
    \Vert \Gamma(u)\Vert_{X^{s_{2,k}}(I)}&\leq \Vert e^{it\Box}u_0\Vert_{X^{s_{2,k}}(I)}+ \Big\Vert \mathcal{I}[F(u)]\Big\Vert_{X^{s_{2,k}}(I)}\\&\leq \Vert u_0\Vert_{H^{s_{2,k}}(\mathbb{R}\times \mathbb{T})}+C\Vert u\Vert_{X^{s_{2,k}}(I)}^{2k+1}\\&\leq \epsilon+C\delta^{2k+1}\leq \delta
\end{align*}
provided that we pick $\epsilon=\frac{\delta}{2}$ and $\delta=(2C)^{-\frac{1}{2k}}$. This shows that $\Gamma$ maps $B_{\delta}$ into itself. Moreover, we execute the same argument as above for $u, v\in B_{\delta}$ by implementing the estimate \eqref{mainest} (with $\delta$ as above) to obtain 
\begin{align*}
    \Vert \Gamma(u)-\Gamma(v)\Vert_{X^{s_{2,k}}(I)}&\leq C(\Vert u\Vert_{X^{s_{2,k}}(I)}^{2k}+\Vert v\Vert_{X^{s_{2,k}}(I)}^{2k})\Vert u-v\Vert_{X^{s_{2,k}}(I)}\\&\leq \frac{1}{2}\Vert u-v\Vert_{X^{s_{2,k}}(I)}
\end{align*}
which proves that $\Gamma$ is a contraction, in particular setting $I=\mathbb{R}_{\geq 0}$, implies the existence of global-in-time solutions to \eqref{eq:HNLS} for small data $u_0\in H^{s_{2,k}}(\mathbb{R}\times \mathbb{T})$. Furthermore, due to the time reversibility of the HNLS equation, the argument carried out above can be extended to the negative axis $I=\mathbb{R}_{\leq 0}$. 

As a consequence of global existence, next, we discuss the scattering for global solutions to \eqref{eq:HNLS} with small initial data. We address such dynamics as $t\rightarrow\infty$ only. By time reversibility, the same argument also holds when $t\rightarrow-\infty$. To this end, it suffices to show that for $u\in X^{s_{2,k}}([0,\infty))$ the following limit exists in $H^{s_{2,k}}(\mathbb{R}\times \mathbb{T})$:
\begin{align}\label{limitincriticalspace}
    \lim_{t\rightarrow \infty}e^{-it\Box}u(t)=u_+\in H^{s_{2,k}}(\mathbb{R}\times \mathbb{T}).
\end{align}
 The argument is fairly straightforward. Indeed, for $u\in X^{s_{2,k}}([0,\infty))$, and any $t_1, t_2>0$ with  $t_1>t_2$, thanks to Proposition \ref{keyproposition}, there exists some $C>0$ such that \begin{align*}
     \Vert u(t_1)-e^{i(t_1-t_2)\Box}u(t_2)\Vert_{H^{s_{2,k}}}=\Vert \mathcal{I}[F(u)](t_1)-\mathcal{I}[F(u)](t_2)\Vert_{H^{s_{2,k}}}\lesssim \Vert u\Vert_{X^{s_{2,k}}((t_1,t_2))}^{2k+1}\leq C
 \end{align*}
 as $k$ is fixed. By this means, as $t_1, t_2\rightarrow\infty$, we conclude that
 \begin{align*}
     \Vert e^{-it_1\Box}u(t_1)-e^{-it_2\Box}u(t_2)\Vert_{H^{s_{2,k}}}=\Vert \mathcal{I}[F(u)](t_1)-\mathcal{I}[F(u)](t_2)\Vert_{H^{s_{2,k}}}\rightarrow 0.
 \end{align*}
Hence, by completeness of $H^{s_{2,k}}$, this implies \eqref{limitincriticalspace}.

Concerning the large initial data, we only consider the case $k\geq 2$, the cubic nonlinearity case $k=1$ follows in a similar way. Thus, in this case, assume that $k$ is fixed. We shall set a time interval $I_T=[0,T]$, $T\leq 1$ (to be determined later). Below we take advantage of the following estimate:
\begin{align}\label{differenceest}
 \Big\Vert \mathcal{I}[F(u+v)-F(u)]\Big\Vert_{X^{s_{2,k}}(I_T)} \lesssim \Vert v\Vert_{X^{s_{2,k}}(I_T)}(\Vert u\Vert_{X^{s_{2,k}}(I_T)}+\Vert v\Vert_{X^{s_{2,k}}(I_T)})^{2k}
\end{align}
which follows from Proposition \ref{keyproposition} after expanding the expression inside the integral. Next, we assume that $u_0\in H^{s_{2,k}}(\mathbb{R}\times\mathbb{Z})$ with $\Vert u_0\Vert_{H^{s_{2,k}}(\mathbb{R}\times\mathbb{Z})}\leq C_0$ for some $0<C_0<\infty$. For a given $N\geq 1$, assume further that $\Vert P_{>N}u_0\Vert_{H^{s_{2,k}}(\mathbb{R}\times\mathbb{Z})}\leq \varepsilon$,
where $P_{>N}f=(Id-P_{\leq N})f$ and $\varepsilon>0$ is a small number to be specified later. In the light of these assumptions, let us define the ball $$B_{\varepsilon,C_0}:=\{u\in X^{s_{2,k}}(I_T)\cap C_t(I_T;H^{s_{2,k}}_{x,y}(\mathbb{R}\times \mathbb{T})): \Vert P_{>N}u\Vert_{X^{s_{2,k}}(I_T)}\leq 2\varepsilon,\,\Vert u\Vert_{X^{s_{2,k}}(I_T)}\leq 2C_0\},$$
on which the operator $\Gamma$ in \eqref{Gammaoperator} will be shown to be a contraction. To this end, let $u\in B_{\varepsilon,C_0}$. By Lemma \ref{linearestlemma}, the estimate \eqref{differenceest}, and Bernstein inequality we obtain \begin{align*}
    \Vert \Gamma(u)\Vert_{X^{s_{2,k}}}&\leq  \Vert e^{it\Box}u_0\Vert_{X^{s_{2,k}}}+\Big\Vert \mathcal{I}[F(P_{\leq N}u)]\Big\Vert_{X^{s_{2,k}}}+\Big\Vert \mathcal{I}[F(u)-F(P_{\leq N}u)]\Big\Vert_{X^{s_{2,k}}}\\&\leq \Vert u_0\Vert_{H^{s_{2,k}}}+C\Vert F(P_{\leq N}u)\Vert_{L^1_tH^{s_{2,k}}_{\mathbf{x}}}+C\Vert P_{>N}u\Vert_{X^{s_{2,k}}}\Vert u\Vert_{X^{s_{2,k}}}^{2k}\\&\leq C_0+CT\Vert P_{\leq N}u\Vert_{L^{\infty}_{t,\mathbf{x}}}^{2k}\Vert P_{\leq N}u\Vert_{L^{\infty}_tH^{s_{2,k}}_\mathbf{x}}+C(2\varepsilon)(2C_0)^{2k}\\&\leq C_0+CTN^2(2C_0)^{2k+1}+C(2\varepsilon)(2C_0)^{2k}\\&\leq 2C_0
\end{align*}
that holds provided that we choose $\varepsilon=\varepsilon(C_0)$ and $T=T(N,C_0)$ sufficiently small. To show the boundedness of $\Gamma$ with respect to the first condition in the definition of the ball $B_{\varepsilon, C_0}$, we split
\begin{align*}
    F(u)=F_1(u)+F_2(u),\quad  F_1(u):=O([P_{>N}u]^2u^{2k-1})\,\,\text{and}\,\,F_2(u):=O([P_{\leq N}u]^{2k}u)
\end{align*}
where the notation $O$ is used to indicate all the terms sharing common factors as stated above, indeed, by writing $u=P_{\leq N}u+P_{> N}u$, we see that $F_1$ and $F_2$ actually consist of combination of terms having factors with complex conjugates and projections $P_{> N}$ or $P_{\leq N}$. Using Lemma \ref{linearestlemma}, Proposition \ref{keyproposition}, Bernstein inequality and the embedding $H^{s_{2,k}}\hookrightarrow L^{2k}$, we obtain
\begin{equation*}
\begin{aligned}
\Vert P_{>N}\Gamma(u)\Vert_{X^{s_{2,k}}}&\leq  \Vert e^{it\Box}P_{>N}u_0\Vert_{X^{s_{2,k}}}+\Big\Vert \mathcal{I}[F_1(u)]\Big\Vert_{X^{s_{2,k}}}+\Big\Vert \mathcal{I}[F_2(u)]\Big\Vert_{X^{s_{2,k}}}\\&\leq \Vert P_{>N}u_0\Vert_{H^{s_{2,k}}}+C\Vert P_{>N}u\Vert_{X^{s_{2,k}}}^2\Vert u\Vert_{X^{s_{2,k}}}^{2k-1}+C\Vert F_2(u)\Vert_{L^1_tH^{s_{2,k}}_{\mathbf{x}}}\\&\leq \varepsilon+C(2\varepsilon)^2(2C_0)^{2k-1}+CT\Vert F_2(u)\Vert_{L^{\infty}_tH^{s_{2,k}}_{\mathbf{x}}}\\&\leq \varepsilon+C(2\varepsilon)^2(2C_0)^{2k-1}\\&\quad+CT\Big[\Vert \langle \nabla \rangle^{s_{2,k}}u\Vert_{L^{\infty}_{t}L^2_{\mathbf{x}}}\Vert P_{\leq N}u\Vert_{L^{\infty}_{t,\mathbf{x}}}^{2k}+N^{s_{2,k}}\Vert u\Vert_{L^{\infty}_tL^{2k}_{\mathbf{x}}}\Vert P_{\leq N}u\Vert_{L^{\infty}_tL^{\frac{4k^2}{k-1}}_{\mathbf{x}}}^{2k}\Big]\\&\leq  \varepsilon+C(2\varepsilon)^2(2C_0)^{2k-1}+CTN^2(2C_0)^{2k+1}\leq 2\varepsilon
\end{aligned}
\end{equation*}
as long as $\varepsilon=\varepsilon(C_0)$ and $T=T(\varepsilon,N,C_0)$ are chosen sufficiently small depending on the indicated parameters. To show that $\Gamma$ is a contraction on $B_{\varepsilon, C_0}$, one needs to deal with the difference $F(u)-F(v)$, which can be estimated by splitting $F$ analogously, so we skip it; concerning the estimate for the difference $F(u)-F(v)$ and uniqueness, see \cite{HHK09, HTT11, KV14}.
\end{proof}

\begin{ackno}\rm
The authors gratefully acknowledge Prof. Yoshio Tsutsumi for his outstanding contributions to the field of dispersive equations, which continue to serve as a profound source of inspiration. Also, we thank the anonymous referees for their detailed suggestions, in particular, for kindly pointing out an error in Theorem 1.8 of the first preprint version of this paper. E.B. and Y.W.~were supported by the EPSRC New Investigator Award 
 (grant no.~EP/V003178/1).
C.S and N.T. were partially supported by the ANR project Smooth ANR-22-CE40-0017.
\end{ackno}
\noindent
\textbf{Declarations}

\noindent
\textbf{Data Availability.} No datasets were generated or analyzed during the current study.

\noindent
\textbf{Competing Interests:} The authors declare that they have no competing interests.

\end{document}